\documentclass[11pt]{amsart}
\scrollmode
\usepackage{amsmath, amsthm, amssymb,amsfonts,mathtools}
\usepackage{verbatim,geometry,graphicx,multirow,array}
\usepackage{algpseudocode}
\usepackage{enumerate}
\usepackage{datetime}

\usepackage{amsmath,amsthm,amssymb,amsfonts}
\usepackage{graphicx, caption, subcaption, listings, float}
\usepackage{multirow,enumitem, hyperref, tabularx}
\usepackage{array,algorithm, algpseudocode}
\usepackage{makecell}
\usepackage{enumitem}

\usepackage{xcolor, mathtools}
\usepackage{tabularx,colortbl}

\newcommand{\x}{{\mathbf x}}
\newcommand{\p}{{\mathbf p}}
\newcommand{\e}{{\mathbf e}}
\newcommand{\y}{{\mathbf y}}
\newcommand{\z}{{\mathbf z}}
\newcommand{\bu}{{\mathbf u}}
\newcommand{\bv}{{\mathbf v}}
\newcommand{\ba}{{\mathbf a}}
\newcommand{\bb}{{\mathbf b}}

\newcommand{\icol}[1]{
	\left(\begin{smallmatrix}#1\end{smallmatrix}\right)
}

\newcommand{\R}{{\mathbb R}}

\newcommand{\Rnxn}{{\R^{n\times n}}}

\newcommand{\bmat}[1]{ \begin{bmatrix}#1\end{bmatrix}}

\newcommand{\Rn}{{\R^n}}

\newcommand{\Paths}{{\mathcal P}}
\newcommand{\LL}{{\mathcal L}}

\newcommand{\itext}[1]{{\qquad\text{#1}\qquad}}
\renewcommand{\ss}{\scriptstyle}

\def\sddots{\mathinner{\raise3pt\vbox{\hbox{$\ss .$}}
		\raise1.5pt\hbox{$\ss .$}\hbox{$\ss .$}}}
\let\hat\widehat
\theoremstyle{plain}
\newtheorem{thm}{Theorem}[section]
\newtheorem{lem}[thm]{Lemma}
\newtheorem{cor}[thm]{Corollary}

\newtheorem{fact}[thm]{Fact}
\newtheorem{summ}[thm]{Summary}
\theoremstyle{definition}
\newtheorem{rem}[thm]{Remark}

\graphicspath{{FiguresALL/}}

\begin{document}

\title{Computing Optimal Trajectories for Optimal Transport in Nonuniform Environments}

\author[Dieci]{Luca Dieci}
\address{School of Mathematics, Georgia Institute of Technology,
Atlanta, GA 30332 U.S.A.}
\email{dieci@math.gatech.edu}
\author[Omarov]{Daniyar Omarov}
\address{Department of Mathematical and Statistical Sciences, University of Alberta, Edmonton, AB T6G 2G1 Canada}
\email{daniyar@ualberta.ca}

\null\hfill {\fontsize{10}{10pt} \selectfont 
Version of \today }


\begin{abstract}
In this work, we solve a discrete optimal transport problem in a nonuniform environment.
To solve the optimal transport problem, we build the cost matrix
and then use classical solvers for discrete optimal transport.  
The challenge is to form the cost matrix, which requires finding the optimal path between two points, and for this task we formulate and solve the associated Euler-Lagrange equations.  A main contribution of ours is to provide verifiable sufficient conditions of optimality of the solution of the Euler-Lagrange equation and to propose new algorithms to to check optimality a-posteriori, thus validating the (exact) computation of the cost matrix. We illustrate our results and performance of the algorithms on several numerical examples in 2 and 3 dimensions.
\end{abstract}

\maketitle

\pagestyle{myheadings}
\thispagestyle{plain}
\markboth{L.~Dieci, D.~Omarov}{Optimal Trajectories for Optimal Transport}

{\bf Notation}.  Vectors are indicated with boldface and matrices with capital letters.  For a symmetric matrix $A$, we will write $A\succ 0$ to signify that $A$ is positive definite.  The standard basis of $\Rn$ is indicated with $\{\e_i\}_{i=1}^n$.  A continuous function $\x:\ [0,1]\to \Rn$ will be called a trajectory, or a path.  The set $\Paths(\ba,\bb)$ will be the set of differentiable trajectories from $\ba\in \Rn$ to $\bb\in \Rn$.  In particular, if $\x\in \Paths(\ba,\bb)$, then $\x(0)=\ba$ and $\x(1)=\bb$.  A smooth and strictly positive real valued scalar function $K(\x)$, from $\Rn$ to $\R^+$, is called a weight (or a kernel).  The norm of a vector $\x$ will always be the Euclidean norm, unless otherwise stated.

\section{Introduction}\label{sec:intro}
Motivated by solving optimal transport problems in the presence of obstacles, in this work we consider solving a discrete optimal transport of masses in a nonuniform environment.  More specifically, we will consider the case where the cost function $c(\ba, \bb)$ to move one unit of mass from location $\ba\in \Rn$ to location $\bb\in \Rn$
is given as the solution of a problem of calculus of variations of the form
\begin{equation}\label{CV}
c(\ba,\bb)= \inf_{\x(t)\in \Paths(\ba,\bb} \int_0^1 \LL(\x,\dot \x) dt\ ,
\end{equation}
and we are interested in the special forms of $\LL$ given in \eqref{length-cost} and 
\eqref{energy-cost}.
From the theoretical point of view, the model \eqref{CV} is a classical problem in Calculus of Variations (see Section \ref{setup} below), and it is considered also by Villani in his book on Optimal Transport, \cite[Chapter 7]{VillaniOldNew}.  Our goal in this work is numerical, that is we want to compute the optimal transport plan of moving $k_0$ point masses $\x_i$'s into $k_1$ point masses $\y_j$'s, where $\x_i$'s and $\y_j$'s are points in some subset of $\R^n$, when the cost of moving one into the other is given by \eqref{CV}.

In a uniform environment, the case that is usually considered in the literature, the transportation cost $c(\x,\y)$ is typically given by a $p$-norm, $c(\x,\y)=\|\x-\y\|_p$, $1<p<\infty$, the $2$-norm $\|\x-\y\|_2$, or simply $\|\x-\y\|$, being the most obvious choice, and also by the $2$-norm square, $c(\x,\y)=\frac12 \|\x-\y\|_2^2$, a case that is known as the $2$-Wasserstein distance, or simply Wasserstein distance.   So, in a uniform environment in $\Rn$, computation of the cost is straightforward.   But, in a nonuniform environment, the cost will generally depend on the path followed to go from $\x$ to $\y$, and this makes the computation of $c(\x,\y)$ a challenge in its own rights.

Optimal transport, in one of its several variants, has been receiving a lot of attention in recent decades, in no small part because of its flexibility to adapt to many problems of seemingly different nature, such as the Schr\"odinger bridge problem, unbalanced optimal transport, transport in the presence of physical constraints, and the use of optimal transport in Machine Learning (e.g., see \cite{PMLR-Neklyudov}).  Our work is particularly motivated by studying optimal transport in the presence of physical constraints, obstacles.
The study of optimal transport in the presence of obstacles has received both theoretical and computational attention; for example, see \cite{Cavalletti} for theoretical results, and \cite{BarrettPrigozhin, OT-convex-obstacle} for computational techniques.  In particular, in \cite{BarrettPrigozhin} Barrett and Prigozhin are concerned with transporting a mass into another, in the plane, by minimizing the cost given by the Euclidean distance when there is an obstacle along straight lines paths, and their computational approach consists of a regularization of the cost, followed by a finite element discretization of the underlying Monge-Ampere PDE.  Presence of an obstacle is one instance of a nonuniform environment, which is the general case we consider, and we do it in the context of a discrete optimal transport problem so that we end up needing to compute optimal trajectories between two points (in the plane or in space) subject to a weighted ``distance'' function.  This very task bears similarities with that of obtaining geodesics on general Riemannian manifolds (e.g., see \cite[Section 2.C]{RiemannGeo}), as we will exemplify in the body of the paper.  

A plan of the paper is as follows.  In Section \ref{setup} below, we review the basic setup of the discrete optimal transport problem and discuss the two choices of cost we consider: the length and the energy costs. In Section \ref{OptPath}, we derive the Euler-Lagrange equations for the cases of interest to us, recall a well known equivalence result between the two cases considered, and discuss its implications for algorithmic development.  We also give an important theoretical result establishing sufficient conditions for optimality; we will exploit this result computationally to rigorously verify optimality.  In Section \ref{algos}, we outline the algorithms adopted in full details, and in Section \ref{results} we present simulation results on several examples, highlighting performance on a number of scenarios.

\subsection{Setup of the problem and our tasks}\label{setup}
At a high level, optimal transport is concerned with moving a probability measure $\mu$, with support $X\subset \Rn$, into another probability measure $\nu$, with support $Y\subset \Rn$, while minimizing a given cost function $c(\x,\y):\  X\times Y \to \R^+$.  
The Monge-Kantorovich (MK) problem consists in determining the optimal transport (or transfer) plan, that is the
joint probability density function $\pi$ $\in {\mathcal{M}}(X \times Y)$, with marginals $\mu$ and
$\nu$, which realizes the $\inf$ of
the functional $J(\pi)$ below:
\begin{equation}\label{InfOT}
	\inf_{\pi} J  := \int_{X\times Y} c(\x,\y) \,d\pi = \int_{X\times Y} c(\x,\y) \,\pi(\x,\y) d\x d\y\,,
\end{equation}
subject to the constraints
$$\mu(\x)=\int_Y\pi(\x,\y)d\y\ ,\quad \nu(\y)=\int_X\pi(\x,\y) d\x\ .$$

\subsubsection{Discrete Optimal Transport}\label{OT}
This is the most commonly studied setup of an optimal transport problem.

Here, $\mu$ and $\nu$ are discrete probability measures on $X = \{\x_i\}_{i=1}^{k_0}$ and $Y = \{\y_j\}_{j=1}^{k_1}$:
\begin{equation}\label{measures}
	\mu := \sum_{i=1}^{k_0}\mu_{i}\delta(\x_i)\ ,\ \ \nu := \sum_{j=1}^{k_1}\nu_{j}\delta(\y_j)\ ,
\end{equation}
where $\mu_i> 0$, $\forall i$, $\nu_j> 0$, $\forall j$, and $\sum_{i=1}^{k_0}\mu_i = \sum_{j=1}^{k_1}\nu_j=1$.

The MK problem effectively becomes a linear programming problem.  Given a cost matrix $C\in \R^{k_0\times k_1}$ of entries $c(\x_i,\y_j)$, 
the optimal transport plan $\pi$ is itself a matrix with nonnegative entries in $\R^{k_0\times k_1}$ that minimizes $J(\pi)$ below:
\begin{equation}\label{DiscreteMK}
	\begin{cases}
		& J(\pi)  := \sum_{i=1}^{k_0}\sum_{j=1}^{k_1} c(\x_i,\y_j) \pi_{ij}\ ,\ \text{subject to }  \\
		&\sum_{j=1}^{k_1} \pi_{ij} = \mu_i,\ i=1,\dots,k_0,\
		\text{and }\sum_{i=1}^{k_0} \pi_{ij} = \nu_j,\ j=1,\dots, k_1\ .
	\end{cases}
\end{equation}
Our numerical method will consist of forming the cost matrix $C=(c(\x_i,\y_j))_{i=1:k_0, j=1:k_1}$ and then using well developed techniques to solve \eqref{DiscreteMK}, namely an assignment solver or the Sinkhorn algorithm,  see Section \ref{DiscOT}.  So doing, we need to effectively form and store the cost matrix.  Of course, storing the matrix $C$ may not be feasible if $k_0$ and $k_1$ are truly large; but, for the values of $k_0$ and $k_1$ we have in mind (say, in the order of thousands) this is not a concern.  As far as forming $C$, this will be discussed in details below.  Presently, we just observe that forming $C$ is a very parallelizable task, since the entries $c(\x_i,\y_j)$ are independent of each other.  So, the cost of computing $C$ is proportional to that of computing one trajectory (assuming that computation of trajectories costs roughly the same) times $k_0k_1$ and divided by the number of processors we have.  Moreover, we want to form $C$ ``exactly'' in nonuniform environments, and this will require carefully validated computation of the $c(\x_i,\y_j)$'s.  
Of course, there are alternatives to our approach that avoid forming $C$ directly. For example, the authors of \cite{SolomonEtAl} -- who are concerned with computer graphics applications and large values of $k_0$ and $k_1$ -- develop a method that considers regularized Wasserstein distances by computing optimal transportation plans through iterative kernel convolutions, by adapting the approach of \cite{BB}; they give impressive graphical displays of their results, but their work lacks a convergence analysis and it does not appear to be immediately applicable to nonuniform environments.
		
\subsubsection{Cost Matrix}\label{WhatCost}
To set up the cost matrix, the first thing to do is 
to select a cost $c(\ba,\bb)$. We consider two costs, the ``length'', and the ``energy'' costs, which generalize the Euclidean and the Wasserstein distances, respectively.

{\bf Agreement}.  In this work, the functional $K$ below is as smooth as required to validate differentiating it as often as needed.

{\bf The length cost}.  This is given by
\begin{equation}\label{length-cost}
c(\ba,\bb)=\min_{\x\in \Paths(\ba,\bb)} \int_0^1 K(\x(t))\  \|\dot \x(t)\| dt  	\  ,
\end{equation}
where the norm is the Euclidean norm, and $K$ is a positive function 
characterizing the environment, that is its geometry.   If $K=1$ (or any fixed constant, for what matters), we have a {\em uniform environment}, otherwise a {\em nonuniform environment}.  Clearly, if $K=1$, we are just looking at the arc-length of the path from $\ba$ to $\bb$, whereas if $K$ is not constant we have a weighted arc-length of the path.  It is a simple exercise in calculus of variations (e.g., see \cite{Liberzon}) that, when $K=1$, the shortest path is a straight line: $\x(t)=\ba+t(\bb-\ba)$, $0\le t\le 1$, so that the cost is precisely the Euclidean distance from $\ba$ to $\bb$: $c(\ba,\bb)=\|\ba-\bb\|$, as one should expect in a uniform environment.  

{\bf The energy cost}  This is given by 
\begin{equation}\label{energy-cost}
	c(\ba,\bb)=\min_{\x\in \Paths(\ba,\bb)} \int_0^1 \frac12 K^2(\x(t))\  \|\dot \x(t)\|^2 dt  	\  .
\end{equation}
In the case of a uniform environment, $K=1$, we are just looking at the average of kinetic energy along the path from $\ba$ to $\bb$.  It is immediate to see that, when $K=1$, the minimum is obtained at the straight line $\x(t)=\ba+t(\bb-\ba)$, $0\le t\le 1$, so that the cost is precisely the Wasserstein distance: $c(\ba,\bb)=\frac12 \|\ba-\bb\|^2$, as we expected in a uniform environment.

\section{Optimal Trajectory}\label{OptPath}
Our main concern in this section is to solve the minimization problem for the optimal trajectory from $\ba$ to $\bb$ for $\x$-dependent weight function $K(\x)$, which we need to do in order to setup the cost matrix. With respect to \eqref{CV}, $\LL(\x,\dot \x)$ is either the integrand in \eqref{length-cost} or \eqref{energy-cost}.
Our task is to find the trajectory $\x\in \Paths(\ba,\bb)$ which gives the minimum value below
\begin{equation}\label{Lagr}
	\min_{\x\in \Paths(\ba,\bb)} \int_0^1 \LL(\x,\dot \x) dt\ .
\end{equation}
Regardless of whether $\LL$ is given by \eqref{length-cost} or by \eqref{energy-cost}, the minimizing path will be called a {\em geodesic}; this naming will be justified by Fact \ref{Equivalence} below.

\subsection{Necessary condition: Euler-Lagrange equations}\label{EL}
The problem formulated in \eqref{Lagr} is a classical problem of {\emph{Calculus of Variations}} and the search for a minimizer follows the approach based on solving the Euler-Lagrange equations, which are first order necessary conditions for obtaining a minimum. Naturally, the derivation of the Euler-Lagrange equations below is a standard textbook procedure, but since we have not explicitly found the precise form of the
equations for the Lagrangian $\LL$ given by \eqref{length-cost} or by \eqref{energy-cost}, and their derivation is short, we provide details.
For \eqref{Lagr}, the Euler-Lagrange equations are:
\begin{equation}\label{EulerLagrange}
	\frac{\partial \LL}{\partial x_i} - \frac{d}{dt}\bigg{[}\frac{\partial \LL}{\partial \dot x_i}\bigg{]} = 0\ ,\  i=1,2,\dots,n\ ,
\end{equation}
or $\sum_{j=1}^n\LL_{\dot x_i\dot x_j}\ddot x_j + \sum_{j=1}^n\LL_{\dot x_ix_j}\dot x_j - \LL_{x_i} = 0$, for all $i$, or more compactly as
\begin{equation}\label{EL-Lagr}
	[D^2_{\dot \x}\LL]\ddot \x + [D_{\dot \x \x}\LL] \dot \x = \nabla_{\x}\LL\ ,
	\quad \x(0)=\ba,\, \ \x(1)=\bb\ ,
\end{equation}
where $D^2_{\dot \x} \LL \in \Rnxn$, $(D^2_{\dot \x} \LL)_{ij} = \LL_{\dot x_i\dot x_j}$, and
$D_{\dot \x \x}\LL \in \mathbb{R}^{n\times n}$, $(D_{\dot \x\x}\LL)_{ij} = \LL_{\dot x_ix_j}$.
We can further simplify \eqref{EL-Lagr} in the two cases of interest to us.

\subsubsection{Euler-Lagrange for length cost}
Here, 
\begin{equation}\label{LagrLength}
\LL(\x,\dot \x)=K(\x)\|\dot \x\|\ .
\end{equation}
\begin{lem}
For \eqref{LagrLength}, the differential equation in \eqref{EL-Lagr} rewrites as
\begin{equation}\label{EL-length-0}
\frac{K(\x)}{\|\dot \x\|^3} \bigg{[}\|\dot \x\|^2I - \dot \x (\dot \x)^T\bigg{]}\ddot \x + \frac{1}{\|\dot \x\|} \bigl((\nabla K)^T\dot \x\bigr) \dot \x = \|\dot \x\|\nabla K\ .
\end{equation}
\end{lem}
\begin{proof}
We have
$\LL_{x_i} = K_{x_i}(\x)\|\dot \x\|$, $\LL_{\dot x_i} = \frac{K(\x)\dot x_i}{\|\dot \x\|}$, $\LL_{\dot x_i x_j} = \frac{K_{x_j}(\x)\dot x_i}{\|\dot \x}$,
$\LL_{\dot x_i\dot x_i} =K(\x)\frac{\sum_{j=1,j\neq i}^n[\dot x_j^2]}{\|\dot \x\|^3}$ and $\LL_{\dot x_i\dot x_j} = -K(\x)\frac{\dot x_i\dot x_j}{\|\dot \x\|^3}$ for $j\neq i$.  Therefore, 
$$	D_{\dot \x\x}\LL=
\frac{1}{\|\dot \x\|}\dot \x(\nabla K)^T \quad \text{and}
\quad  D^2_{\dot \x}\LL = \frac{K(\x)}{\|\dot \x\|^3} \bigg{[}\|\dot \x\|^2I - \dot \x (\dot \x)^T\bigg{]}\ . $$
Using these in \eqref{EL-Lagr}, we obtain \eqref{EL-length-0}.
\end{proof}

The difficulty now is that in \eqref{EL-length-0} the symmetric matrix multiplying $\ddot \x$ is singular, so it cannot be just inverted to obtain a differential equation for $\ddot \x$.   In fact, we have the following result.
\begin{lem}\label{LinAlg}
The symmetric matrix $\frac{K(\x)}{\|\dot \x\|^3} \bigl[ \|\dot \x\|^2I - \dot \x (\dot \x)^T\bigr]$ has a simple eigenvalue equal to $0$, associated to the eigenvector $\dot \x$, all other eigenvalues being equal to $\frac{K(\x)}{\| \dot \x\|}$.
\end{lem}
\begin{proof}
The fact that $0$ is an eigenvalue associated to the eigenvector $\dot \x$ is an immediate verification.  The other part of the result follows by considering the linearly independent vectors $\dot x_j\e_1-\dot x_1 \e_j$ and verifying  that each of these is an eigenvector with eigenvalue $\frac{K(\x)}{\| \dot \x\|}$.
\end{proof}

To rewrite \eqref{EL-length-0} in a way that gives a well defined boundary value problem, we need to restrict to the orthogonal complement of $\dot \x$.
The algebra is lengthy, but not difficult, and in the end we obtain the following result.

\begin{lem}
If \eqref{EL-length} below is satisfied, then so is \eqref{EL-length-0}, and thus so is \eqref{EL-Lagr} for the cost \eqref{LagrLength}:
\begin{equation}\label{EL-length}
	\ddot \x = \frac{\|\dot \x\|^2}{K(\x)} \nabla K \ ,\quad \x(0)=\ba,\,\ \x(1)=\bb\ .
\end{equation}
\end{lem}
\begin{proof}
Substituting $\ddot \x$ from \eqref{EL-length} into \eqref{EL-length-0}, and simplifying, yields the result.
\end{proof}

\subsubsection{Euler-Lagrange for energy cost}
Here 
\begin{equation}\label{LagrEner}
\LL(\x,\dot \x)=\frac12 K^2(\x) \|\dot \x\|^2\ ,
\end{equation}
and it is easier to rewrite the boundary value problem \eqref{EL-Lagr}. 
\begin{lem}
For \eqref{LagrEner}, the system \eqref{EL-Lagr} is
\begin{equation}\label{EL-energy}
    \ddot \x = \frac{\|\dot \x\|^2}{K(\x)} \nabla K - \frac{2[\nabla K]^T \dot \x}{K(\x)}\dot \x \ ,\quad \x(0)=\ba,\,\ \x(1)=\bb\ .
\end{equation}
\end{lem}
\begin{proof}
We have
$\LL_{x_i} = K(\x)K_{x_i}(\x)\|\dot \x\|^2$, $\LL_{\dot x_ix_j} = 2K(\x)K_{x_j}(\x)\dot x_i$, and $\LL_{\dot x_i} = K^2(\x)\dot x_i$, so that
$\LL_{\dot x_i\dot x_i} =K^2(\x)$ and $\LL_{\dot x_i\dot x_j} = 0$ for $j\neq i$.  It follows that 
$$D_{\dot \x\x}\LL =2K(\x)\dot \x(\nabla K)^T,$$
so that the differential equation in \eqref{EL-Lagr} can be rewritten as:
$$K^2(\x) \ddot \x + 2K(\x)\dot \x[\nabla K]^T \dot \x = K(\x)\|\dot \x\|^2 \nabla K\ ,$$
which is exactly 
$$\ddot \x = \frac{\|\dot \x\|^2}{K(\x)} \nabla K - \frac{2(\nabla K)^T \dot \x}{K(\x)}\dot \x\ ,$$
as claimed.
\end{proof}
An interesting and useful fact about solutions of \eqref{EL-energy} is that they have constant speed.
\begin{cor}\label{const-speed}
If the trajectory $\y(t)$ is a solution of the Euler-Lagrange equation
\eqref{EL-energy}, then $\y(t)$ has constant speed:
$$K(\y) \| \dot \y\| = \text{constant}\ .$$
\end{cor}
\begin{proof}
Since $\y$ satisfies
$$\ddot \y = \frac{\|\dot \y\|^2}{K(\y)} \nabla K(\y) - \frac{2[\nabla K(\y)]^T \dot \y}{K(\y)}\dot \y\ ,$$
then, taking the dot-product of this relation with $K\dot \y$, we obtain
$$K\ddot \y^T\dot \y = - \|\dot \y\|^2[\nabla K]^T \dot \y\quad \text{or} \quad \frac{d}{dt}[K^2\|\dot \y\|^2] = 0\ , $$
as claimed.
\end{proof}
\begin{rem}\label{const-speed-rem}
In our numerical algorithms, we use Corollary \ref{const-speed} as an a-posteriori check to infer whether we have effectively solved \eqref{EL-energy}; see Sections \ref{algos} and \ref{results}.
\end{rem}

\subsection{Sufficient conditions: Conjugate points}\label{Sec:Riccati} 
As it is well understood, satisfying the Euler-Lagrange equation 
\eqref{EL-Lagr} gives an extremal path, but it is only a necessary condition for a minimizer.  The present concern is to obtain sufficient, and verifiable, conditions for having obtained a minimizer in \eqref{Lagr}.  The original work of  Gelfand and Fomin, \cite{Gelfand}, lays the groundwork for providing sufficient conditions by looking at positivity of the second variation.  But unfortunately the derivation of sufficient conditions in \cite{Gelfand} requires the matrix $\LL_{\x\dot \x}$ to be symmetric, which is not satisfied in our case; the same assumption is also made in more recent books on the topic (e.g., see \cite[p.19]{Jost} and \cite[p.54]{Liberzon}).  Sufficient conditions without requiring symmetry of $\LL_{\x\dot \x}$ are given by Zelikin in \cite{Zelikin}, as we will recall below.  To formulate these conditions, let us look at the 2nd variation of the functional $\LL$ along a trajectory $\x(t)$ that solves the Euler-Lagrange equation, in the cases of interest to us.  We further elucidates how the sufficient conditions can be constructively verified.

For the 2nd variation, we have
\begin{equation}\begin{split}\label{2ndVar0}
		\delta^2 \LL = \frac{1}{2} & \int_0^1\left( \y^T\LL_{\x\x} \y+\dot \y^T \LL_{\dot \x \dot \x}\dot \y +\dot \y^T\LL_{\dot \x \x}\y +\y^T\LL_{\x\dot \x}\dot \y\right)dt \\
		:= \frac{1}{2} & \int_0^1\left( \y^TA \y+\dot \y^T C\dot \y +\dot \y^TB^T \y +\y^TB\dot \y\right) dt\ , \\
		\text{where} & \quad  A=\LL_{\x\x},\ B=\LL_{\x\dot \x}, \ C=\LL_{\dot \x \dot \x}\ ,
\end{split}\end{equation}
and where $\y$ is a differentiable function from $[0,1]$ to $\Rn$ satisfying homogeneous boundary conditions, $\y(0)=\y(1)=\mathbf{0}$.
We rewrite \eqref{2ndVar0} in the compact form
\begin{equation}\label{2ndVar1}
	\delta^2 \LL =\frac12 \int_0^1\bmat{\y^T & \dot \y^T} \bmat{A & B \\ B^T & C} \bmat{\y \\ \dot \y} dt \ , \,\ \y(0)=\y(1)=\mathbf{0}\ ,
\end{equation}
and we stress that $A=A^T$ and $C=C^T$, but in general $B\ne B^T$.
According to classical results, we need the 2nd variation to be a positive quantity in order for the  extremal path $\x(t)$ to be a minimizer of the functional $\int_0^1 \LL(\x,\dot \x) dt$. The following rewriting of \eqref{2ndVar1} is useful to highlight how, and where, the lack of symmetry of $\LL_{\x\dot \x}$ shows up.

\begin{lem}\label{SecondVar}
The 2nd variation in \eqref{2ndVar1} rewrites as
\begin{equation}\label{2ndVar2}\begin{split}
\delta^2 \LL & =\frac12 \int_0^1 \bmat{\y^T & \dot \y^T} \bmat{Q & E \\  E^T & P} \bmat{\y \\ \dot \y} dt\ , \\ 
\text{where} & \,\ P=P^T=C\ , \,\  
Q=Q^T=A-\frac{\dot B+\dot B^T}{2}\ ,\,\ E=\frac{B-B^T}{2}\ . 
\end{split}\end{equation}
\end{lem}

\begin{proof}
The following two identities will be handy:
\begin{equation}\begin{split}\label{derivatives}
		& \text{(a)}\quad \dot \y^TB^T\y=\frac{d}{dt}(\y^TB^T\y)-\y^T\dot B^T\y-\y^TB^T\dot \y\ , \\
		& \text{(b)}\quad \y^TB \dot \y=\frac{d}{dt}(\y^TB\y)-\y^T\dot B\y-\dot \y^TB \y \ .
\end{split}\end{equation}
Now, using \eqref{derivatives}-(a) in \eqref{2ndVar0} and recalling that $\y(0)=\y(1)=0$, we obtain $$\delta^2 \LL = \frac12 \int_0^1\left(\dot \y^T C\dot \y + \y^T(A-\dot B^T)\y+ \y^T(B-B^T)\dot \y\right)dt\ , $$ and since $\y^T(B-B^T)\dot \y=\dot \y^T(B^T-B) \y$, this rewrites as
\begin{equation}\label{2ndVar2a}
\delta^2 \LL =\frac12 \int_0^1 \bmat{\y^T & \dot \y^T} \bmat{A-\dot B^T & (B-B^T)/2\\  (B^T-B)/2 & C} \bmat{\y \\ \dot \y} dt\ .
\end{equation}
Similarly, using \eqref{derivatives}-(b) in \eqref{2ndVar0} and recalling that $\y(0)=\y(1)=0$, we obtain 
$$\delta^2 \LL = \frac12 \int_0^1\left(\dot \y^T C\dot \y + \y^T(A-\dot B)\y+ \dot \y^T(B^T-B) \y\right)dt\ ,$$
and since $\dot \y^T(B^T-B) \y=\y^T(B-B^T)\dot \y$, this rewrites as
\begin{equation}\label{2ndVar2b}
	\delta^2 \LL = \frac12 \int_0^1 \bmat{\y^T & \dot \y^T} \bmat{A-\dot B & (B^T-B)/2\\  (B-B^T)/2 & C} \bmat{\y \\ \dot \y} dt\ .
\end{equation}
Finally, adding together \eqref{2ndVar2a} and \eqref{2ndVar2b}, and dividing by $2$, we get the expression
\begin{equation}\begin{split}
		\delta^2 \LL & = \frac12 \int_0^1 \bmat{\y^T & \dot \y^T} \bmat{A-\frac{\dot B+\dot B^T}{2} & \frac{B-B^T}{2} \\  \frac{B^T-B}{2} & C} \bmat{\y \\ \dot \y} dt\ ,
\end{split}\end{equation}
as claimed.
\end{proof}
\begin{rem}
The
expression \eqref{2ndVar2} is what we obtain in case $B\ne B^T$, and there is the symmetric part of $\dot B$ showing up in $Q$, as well as the antisymmetric part of $B$ in $E$.  Recall that $\dot B=\frac{d}{dt}\LL_{\x\dot \x}$.  Of course, in case $B=B^T$, then $Q=A-\dot B$ and $E=0$.
\end{rem}

With Lemma \ref{SecondVar}, we can now proceed to give verifiable sufficient conditions for having a minimizer.  We need $\delta^2 \LL$ to be positive when considering all smooth\footnote{piecewise smooth would actually suffice} functions $\y$ such that $\y(0)=\y(1)=0$. The following fundamental result gives the sufficient conditions, see \cite{Gelfand}, and is given without proof.

\begin{thm}\label{GelfandSuff}
The second variation $\delta^2 \LL$ in \eqref{2ndVar2} is positive if 
\begin{itemize}
	\item[(i)] $P\succ 0$ (Legendre condition), where $P=\LL_{\dot \x\dot \x}$, and 
	\item[(ii)] there are no conjugate points to $0$ in $(0,1]$.  
\end{itemize}
\qed
\end{thm}

We can verify directly condition (i) of Theorem \ref{GelfandSuff}, in the cases of interest to us.
\begin{itemize}
	\item[(a)] In the case of cost given by length, \eqref{length-cost}, because of Lemma \ref{LinAlg}, we have $\LL_{\dot \x\dot \x}=
	\frac{K(\x)}{\|\dot \x\|^3} \bigl[ \|\dot \x\|^2I - \dot \x (\dot \x)^T\bigr]$ which is only positive semidefinite, hence condition (i) of Theorem \ref{GelfandSuff} fails.
	\item[(b)]  In the case of cost given by energy, \eqref{energy-cost}, we have
	$\LL_{\dot \x\dot \x}=K(\x)^2 I$ so (i) of Theorem \ref{GelfandSuff} holds.
\end{itemize}

As far as the absence of conjugate points, when $P\succ 0$, the absence of conjugate points can be characterized by passing to the Hamiltonian formulation of \eqref{2ndVar2} and from there to the Riccati equation in \eqref{Riccati}.    The underlying process is well understood, and it is the one laid out by Gelfand-Fomin in \cite{Gelfand} and Zelikin in \cite{Zelikin}.  We derive it in our case.

\begin{thm}\label{SuffRiccati}
Let $P$ be positive definite.  Then, the absence of conjugate points is equivalent to requiring that the $(n,n)$ matrix $U$ in \eqref{LinHam} is invertible on $(0,1]$.  Here, we let
$Y=\bmat{U \\ V} \in \R^{2n,n}$, where $U,V\in \Rnxn$ are given by the solution of the following linear system
\begin{equation}\label{LinHam}
	\dot Y=MY=\bmat{-P^{-1}E^T & P^{-1} \\ Q-EP^{-1}E^T & EP^{-1}}\bmat{U\\ V}\ ,\,\ \bmat{U\\ V}(0)=\bmat{0\\ P(0)}\ ,
\end{equation}
and all blocks are $(n,n)$. \\
Alternatively, the absence of conjugate points is guaranteed if the  
Riccati equation in \eqref{Riccati} has a well defined symmetric and bounded solution over the entire interval $(0,1]$, where
\begin{equation}\label{Riccati}	
    \dot W= Q-EP^{-1}E^T +WP^{-1}E^T+EP^{-1}W-WP^{-1}W \ .
\end{equation}
\end{thm}

\begin{proof}
The first step is to take the Euler-Lagrange equation of the functional in the 2nd variation in \eqref{2ndVar2}.  We have
$$\delta^2 \LL  =\frac12 \int_0^1 \left( \dot \y^T P \dot \y+ \dot \y^TE^T\y+\y^TE\dot \y+\y^T Q \y \right) dt \ ,$$
and so the Euler-Lagrange equation reads
$$(Q\y+E\dot \y)-\frac{d}{dt}(P\dot \y+E^T\y)\ = \ 0\ .$$
Then, we change variables using  the (generalized) momenta $\p=P\dot \y+E^T\y$, 
so that $\dot \y=P^{-1}\p -P^{-1}E^T\y$ and from these we obtain the sought linear system (see \eqref{LinHam}):
$$\begin{cases} \frac {d}{dt}\y & = - P^{-1}E^T \y +P^{-1}\p \\
	\frac {d}{dt} \p &= (Q-EP^{-1}E^T)\y + EP^{-1}\p \ .
\end{cases}$$
With this rewriting, the fact that the absence of conjugate points is equivalent to invertibility of $U$ is a classic result (see \cite[Def.1, Ch.5 and note that $V(0)=P(0)\dot Y(0)$ there]{Gelfand}).\\
The connection of this fact to the existence of a solution to \eqref{Riccati} derives from the realization that, for as long as $U$ is invertible, one has that
the solution of \eqref{Riccati} is given by
$$W(t)=V(t)U^{-1}(t)\ .$$
This latter result is easily verified directly.  In fact, 
\begin{gather*}
    \frac{d}{dt} \bigl(V(t)U^{-1}(t)\bigr)=\dot V U^{-1}-VU^{-1}\dot U U^{-1}= \\
    Q-EP^{-1}E^T +(VU^{-1})P^{-1}E^T+EP^{-1}(VU^{-1})-(VU^{-1})P^{-1}(VU^{-1})  \ , 
\end{gather*}
so that $W$ and $VU^{-1}$ satisfy the same differential equation.
\end{proof}

\begin{rem}
The matrix $M$ in \eqref{LinHam} is a special case of a {\sl Hamiltonian matrix}.  Recall that a matrix $M=\bmat{M_{11} & M_{12} \\ M_{21} & M_{22}}$ is Hamiltonian if $M_{12}$ and $M_{21}$ are symmetric and $M_{22}=-M_{11}^T$.
\end{rem}

\begin{summ}\label{Summ1}
For the two costs given by \eqref{length-cost}, respectively \eqref{energy-cost}, we can solve the Euler-Lagrange equations \eqref{EL-length}, respectively \eqref{EL-energy}.  Based on Lemma \ref{const-speed} and Theorem \ref{GelfandSuff}, in the case of \eqref{energy-cost}, we have reliable criteria to establish both that we have actually satisfied the Euler-Lagrange equation and that the computed solution is a minimizer; we will exploit these facts algorithmically, see Section \ref{algos} and \ref{results}.  However, in the case of cost given by \eqref{length-cost}, we do not yet have criteria guaranteeing the same.
\end{summ}

\subsection{Equivalence between length and energy minimization}\label{Equivalence}
Powerful results from Riemannian geometry allow bypassing the difficulties pointed out in Summary \ref{Summ1} relative to the cost \eqref{length-cost}.  In fact, the following fundamental result holds; e.g., see \cite[Proposition 2.97]{RiemannGeo} or \cite[Lemma 2.1.1]{Jost}.

\begin{fact}\label{equivalence}
If the energy cost \eqref{energy-cost} is minimized along the curve $\gamma$, then  also the length cost \eqref{length-cost} is minimized along the same curve $\gamma$. Furthermore, we have the following relation between energy and length: $L^2(\gamma) = 2E(\gamma)$.
\end{fact}

The practical impact of Fact \ref{equivalence} is that one does not need to solve the length minimization problem, but solving \eqref{EL-energy}, 
and verifying that the sufficiency conditions hold,  
automatically gives the curve of shortest length.  In this respect, both approaches give the same geodesic curve. At the same time, it has to be emphasized that the parameterizations of the curve are different; that is, the function $x(t)$, $0\le t\le 1$, that solves the Euler-Lagrange \eqref{EL-energy} does not generally also solve \eqref{EL-length}.

\section{Algorithms used}\label{algos}

\subsection{Solving the BVP}\label{SolveBVP}
Solving the BVP \eqref{EL-Lagr} is a crucial component of our algorithm.  We used a collocation approach coupled with homotopy continuation.  We ended up using collocation at Lobatto points as implemented in the software {\tt bvpc5} of {\tt Matlab}.  A key component of collocation methods is the solution of (large) nonlinear systems by use of Newton's method, and this turned out to be a very delicate task and required us to adopt a homotopy approach when solving \eqref{EL-Lagr}.  
Namely, we embedded the problem into a parameter dependent one for which the solution at the $0$-value of the parameter is a straight line.  To clarify, 
when solving both \eqref{EL-length} and \eqref{EL-energy}, we used the homotopy
\begin{equation}\label{HomoLength}
\widehat{K}(x(t,\alpha),\alpha)=(1-\alpha)+\alpha K(x(t,\alpha))\ ,\,\, 0\le \alpha \le 1\ ,
\end{equation}
progressively increasing $\alpha$ from $0$ to $1$ (where we obtain the sought solution of the BVP).  See Figures \ref{fig:Path_E1}, \ref{fig:Path_E2}, and \ref{fig:Path_E3} in Section \ref{results}.

\subsection{A posteriori verification of minimum: Riccati equation}\label{A-poste} To verify that there are no conjugate points in $(0,1]$, we propose the following simple algorithm. Since the blocks of the matrix $M$ are available at the grid points found while solving the BVP, we adopt a scheme that only requires this information.  Below, we refer to $\hat \x(t_k)$ as the computed trajectory at the gridpoints.

\begin{itemize}
\item[(i)] Let $\hat \x(t_k)$ be the computed trajectory at gridpoints $t_k$, $k=0,1,\dots, N$, where $t_0=0$, $t_N=1$.  Form
$A,B,C$ and $\dot B$, at $t_k$, $k=0,1,\dots, N-1$.
To compute $\dot B$, we use forward differences: $\dot B(t_k)=\frac{B(t_{k+1})-B(t_k)}{t_{k+1}-t_k}$, so that we have $\dot B$ at all grid points except the last one.  Then, we form the blocks $P$, $Q=A-\frac{\dot B+\dot B^T}{2}$ and $E=\frac{B- B^T}{2}$ at all grid points, except $t_N$.
\item[(ii)]
Approximate the solution of \eqref{Riccati} by using Euler method on the linear system identified in Theorem \ref{SuffRiccati}. That is, with $U_0=0$ and $V_0=P(0)$, compute 
\begin{equation*}\begin{split}
&  \bmat{U_{k+1} \\ V_{k+1}}\ = \ \bigl(I+h_kM(t_k)\bigr) \bmat{U_{k} \\ V_{k}}\ ,\,\,\text{where}  \\ 
& h_k=t_{k+1}-t_k ,\,\ k=0,1,\dots, N-1, \,\ t_0=0\ ,\, t_N=1.
\end{split}\end{equation*}
Of course, no inverse $P^{-1}$ is explicitly computed when forming $M$, and a linear system is instead solved.   To clarify, we have
\begin{equation*}\begin{split}
&  U_{k+1} \ = \ U_k + h_k \left[-(P^{-1}E^T)(t_k)U_{k} +P^{-1}(t_k)V_k\right] \,\ \text{and} \\
&  V_{k+1} \ = \ V_k + h_k \left[(Q-EP^{-1}E^T)(t_k)U_{k} +EP^{-1}(t_k)V_k\right] \ ,
\end{split}\end{equation*}
and so we solve the two systems $PX=E^T$ and $PY=V_k$.
\item[(ii)] 
Observe that we do not need to form the approximation to \eqref{Riccati}
$W_{k+1}=V_{k+1}U_{k+1}^{-1}\ $, since all we care about is whether or not there are conjugate points, but not the solution of \eqref{Riccati}. 
So, we monitor if $\det(U_{k+1})$ becomes negative, in order to rule out having missed singularities.  At the beginning, $\det(U_0)=0$, and -since we are using Euler method- we have $U_1=h_0I$, which is positive definite.
If there is a change of sign in the determinant of $U_{k+1}$ (step (ii) above), then we declare that there are conjugate points in $(0,1]$, otherwise we do not.  If there are no conjugate points, then the extremal we found is a minimizer.
\end{itemize}

\subsection{Solving the discrete OT}\label{DiscOT}
The formulation \eqref{DiscreteMK} lends itself to well developed algorithms.  We mention two of them which we have used in our numerical experiments below: (a) the linear assignment problem (matching problem), and (b) the Sinkhorn method.

\begin{itemize}
\item[(a)] In the linear assignment problem, given the cost matrix $C$, we match a row to  a column so to minimize the cost.  To illustrate, in \eqref{measures},
take $k_0=k_1=k$ and $\mu_i = \frac{1}{k}$, $\nu_j = \frac{1}{k}$, $\forall i,j$.  In this case, 
$\pi$ will need to be a non-negative doubly stochastic matrix up to a scalar multiplication.  Popular solution techniques for the assignment problem are the Hungarian algorithm (see \cite{kuhn1955hungarian, munkres1957algorithms, edmonds1972theoretical, tomizawa1971some}),  relaxation algorithms like the  auction algorithm of 
Bertsekas and variations of it (\cite{bertsekas1981new, walsh2019real}), as well as Dijkstra algorithm, \cite{Dijkstra}, which is what we used in our numerical experiments in Section \ref{results} as implemented in the builtin function {\tt matchpairs} of {\tt Matlab}, see \cite{DuffKoster}.  We note that --with this technique-- we cannot split masses.
\item[(b)] Sinkhorn.  This relatively recent technique, \cite{Cuturi2013a}, has won the favor of many people, because of the flexibility to handle the general setup of \eqref{measures}.
First, one regularizes the minimization problem by adding to it an {\emph {entropic regularization}} term to smooth out the problem (actually, one adds the negative of entropy regularization).  That is, for  a fixed $\epsilon>0$, one seeks the minimum of $J_{\epsilon}$ below:
\begin{equation}\label{DiscreteMKeps}
	\begin{cases}
		& J_{\epsilon} := \sum_{i=1}^{k_0}\sum_{j=1}^{k_1} c(\x_i,\y_j) \pi_{ij} + \epsilon \sum_{i=1}^{k_0}\sum_{j=1}^{k_1}\pi_{ij}\log\pi_{ij}\ ,\\
		& \text{subject to } \sum_{j=1}^{k_1} \pi_{ij} = \mu_i,\  i=1,\dots,k_0\ , \text{ and } \sum_{i=1}^{k_0} \pi_{ij} = \nu_j,\ j=1,\dots,k_1\ .
	\end{cases}
\end{equation} 
An appealing feature of the method is that it is simple to implement an algorithm to find the optimal plan for the regularized problem \eqref{DiscreteMKeps}. Defining $\bu \in \mathbb{R}^{k_0}$, $\bv \in \mathbb{R}^{k_1}$, and $K\in \mathbb{R}^{k_0\times k_1}$ with $K_{ij}=e^{-\frac{c(\x_i,\y_j)}{\epsilon}}$,
to recover the optimal plan $\pi^*_{\epsilon}$ of \eqref{DiscreteMKeps}, for fixed $\epsilon>0$, one needs to perform a sufficient number of iterations on the vectors $\bu$ and $\bv$ to converge to $\pi^*_{\epsilon}$ (see \cite{Cuturi2013a}):
\begin{equation}\label{SinkhornIter}
	\bv^{(k+1)}_j = \frac{\nu_j}{\{K^T\bu^{(k)}\}}_j ,\,
	\bu^{(k+1)}_i = \frac{\mu_i}{\{K\bv^{(k+1)}\}}_i ,\, \pi_{\epsilon}^{(k+1)} = {\tt{diag}}(\bu^{(k+1)})K{\tt{diag}(\bv^{(k+1)})}\ .
\end{equation}
Observe that, with this technique, we may likely end up splitting masses.
\end{itemize}

\begin{rem}\label{sinkhorn-eps}
In principle, as $\epsilon$ decreases, the minimum value of $J_{\epsilon}$ in \eqref{DiscreteMKeps} approaches the minimum value of $J$ in \eqref{DiscreteMK} at a linear rate. However, in practice, there is always some value $\epsilon_0$, which depends on the problem, such that for $\epsilon<\epsilon_0$, there are numerical stability issues;
because of the $\frac{1}{\epsilon}$ term; see \cite{PeyreCuturi} and \cite{Schmitzer} for a discussion of this aspect, and see Section \ref{results} for the values of $\epsilon$ we used in our examples.
\end{rem}

\section{Numerical results}\label{results}
In this last Section we show results of numerical experiments in 2 and 3 space dimensions.  First, we show results on computation of an optimal trajectory for both lengths and energy costs, exemplifying how we verify optimality as well as equivalence of the two curves we obtain.  Then, we compute the cost matrix for a discrete Optimal Transport problem and solve the latter both using the linear assignment (matching) formulation, and the Sinkhorn method.

\subsection{Optimal Trajectory}\label{Opt_Path}
Below, we show computation of optimal trajectories for the energy and length minimization problems in the context of two-point configurations.   
Our scope here is to verify that we satisfy both necessary and sufficient conditions for the energy-minimizing solution and also to show that the curves obtained by energy-minimization or length-minimization are indeed the same curve. In the  numerical examples below, we used $21$ equispaced homotopy steps (see Section \ref{SolveBVP}) and in all of our experiments we used absolute and relative error tolerances of $10^{-4}$ for solving the BVP. Table \ref{tab:Cost_time} reports the optimal transport cost for each problem discussed below, together with the total computation time, on a fixed mesh of $10^4$ time-steps. From Table \ref{tab:homopoty_time}, we note that the computation time of each homotopy step is roughly the same, which confirms that the numerical method is effectively working with a fixed mesh.

\renewcommand\theequation{E\arabic{equation}} \setcounter{equation}{0}

The following test cases in 2-d and 3-d are considered:
\begin{gather}
	\ba = \icol{-2\\1}\ ,\ \bb = \icol{2\\0}\ ,\ K(\x) = \frac{1}{\frac12+||\x||}\ ,\label{E1}\\
	\ba = \icol{-7\\-5}\ ,\ \bb = \icol{6\\7}\ ,\ K(\x) = \sin(x_1(t)) - \sin(x_2(t)) + 3\ ,\label{E2}\\
	\ba = \frac45\icol{1\\1\\-1}\ ,\ \bb = \frac45\icol{1\\1\\1}\ ,\ K(\x) = ||\x|| + \frac{1}{10}\ .\label{E3}
\end{gather}

\begin{table}[ht]
	\centering\scalebox{1}{\renewcommand*{\arraystretch}{1.5}
		\begin{tabular}{||c||c|c||c|c||}\hline\hline
			& \multicolumn{2}{|c||}{\textbf{Energy Cost}} & 
			\multicolumn{2}{|c||}{\textbf{Length Cost}}\\\hline
			\textbf{Example}  & \textbf{Solution} \eqref{energy-cost} & \textbf{Time} & \textbf{Solution} \eqref{length-cost} & \textbf{Time}\\\hline
			\eqref{E1}  & $2.2917$ & $23.63 $ sec. & $2.1409$ & $20.422$ sec.\\\hline
			\eqref{E2}  & $1108.4$ & $23.566$ sec. & $47.082$ & $23.345$ sec.\\\hline
			\eqref{E3}  & $1.9684$ & $24.796$ sec. & $1.9841$ & $23.458$ sec.\\\hline
			\hline
	\end{tabular}}
	\caption{Minimal cost values and computation time}
	\label{tab:Cost_time}
\end{table}

\begin{table}[ht]
	\centering\scalebox{1}{\renewcommand*{\arraystretch}{1.5}
		\begin{tabular}{||c||c|c|c||c|c|c||}\hline\hline
			& \multicolumn{3}{|c||}{\textbf{Energy Cost}} & 
			\multicolumn{3}{|c||}{\textbf{Length Cost}}\\\hline
			\textbf{Homotopy Step}  & \eqref{E1} & \eqref{E2} & \eqref{E3} & \eqref{E1} & \eqref{E2} & \eqref{E3}\\\hline
			$\alpha = 0   $ & $0.3659$ & $0.45082$ & $0.55204$ & $0.42017$ & $0.50373$ & $0.41298$ \\\hline
            $\alpha = 0.1 $ & $1.1119$ & $ 1.4195$ & $ 1.1878$ & $0.97861$ & $ 1.0408$ & $  1.131$ \\\hline
            $\alpha = 0.2 $ & $1.1061$ & $ 1.2395$ & $ 1.2002$ & $ 1.0344$ & $ 1.2453$ & $ 1.0631$ \\\hline
            $\alpha = 0.3 $ & $1.1176$ & $ 1.2487$ & $ 1.2254$ & $ 1.0008$ & $ 1.2670$ & $ 1.1905$ \\\hline
            $\alpha = 0.4 $ & $1.1128$ & $ 1.2448$ & $  1.213$ & $0.95345$ & $ 1.2678$ & $ 1.1444$ \\\hline
            $\alpha = 0.5 $ & $1.1201$ & $ 1.2477$ & $  1.234$ & $0.95799$ & $ 1.2678$ & $ 1.1459$ \\\hline
            $\alpha = 0.6 $ & $1.2048$ & $ 1.0189$ & $ 1.2004$ & $0.95029$ & $ 1.0441$ & $ 1.1423$ \\\hline
            $\alpha = 0.7 $ & $1.2048$ & $ 1.0241$ & $ 1.2017$ & $   1.02$ & $ 1.0424$ & $ 1.1704$ \\\hline
            $\alpha = 0.8 $ & $1.1836$ & $ 1.0355$ & $ 1.1817$ & $ 1.0197$ & $ 1.0451$ & $ 1.1794$ \\\hline
            $\alpha = 0.9 $ & $1.1707$ & $ 1.0319$ & $ 1.1827$ & $ 1.0337$ & $ 1.0549$ & $ 1.1902$ \\\hline
            $\alpha = 1   $ & $1.3944$ & $ 1.0280$ & $ 1.2024$ & $ 1.0323$ & $ 1.0484$ & $ 1.2687$ \\\hline\hline
            \textbf{Total Time:} & $\mathbf{23.63}$ & $\mathbf{23.566}$ & $\mathbf{24.796}$ & $\mathbf{20.422}$ & $\mathbf{23.345}$ & $\mathbf{23.458}$ \\\hline
			\hline
	\end{tabular}}
	\caption{Computation time (in seconds) across 21 homotopy steps}
	\label{tab:homopoty_time}
\end{table}

\noindent{\bf Example \eqref{E1}}.
In Figure \ref{fig:Path_E1}, we display the energy- and length-minimizing solutions corresponding, along with curves obtained for selected homotopy steps; by virtue of our choice of homotopy \eqref{HomoLength}, of course the curves we obtain should be the same at $\alpha=1$, but not at the intermediate values of $\alpha$.  
On the same mesh where the trajectories are computed, we verify all of the claimed necessary and sufficient conditions.  Namely, 
in Figure \ref{fig:Cond_E1} we verify that the energy-minimizing path has constant-speed (necessary condition) and --by the technique in section \ref{A-poste}-- that there are no conjugate points (sufficiency condition), thereby confirming that we obtained a true minimizer. Further, the difference between the squared length and twice the energy is negligible, as expected, as is the comparison of the lengths of the two paths:
$$
\left|L(\z)^2 - 2E(\y)\right| = 1.22 \times 10^{-8}\ , \itext{and}
\left|L(\z) - L(\y)\right| = 2.85 \times 10^{-9}\ ,
$$
where $\z$ and $\y$ denote the length- and energy-minimizing paths, respectively, and $L(\z)$ and $E(\y)$ are their corresponding length and energy values.

\begin{figure}[ht]
	\centering
	\begin{subfigure}{.49\textwidth}
		\centering
		\includegraphics[width=\linewidth]{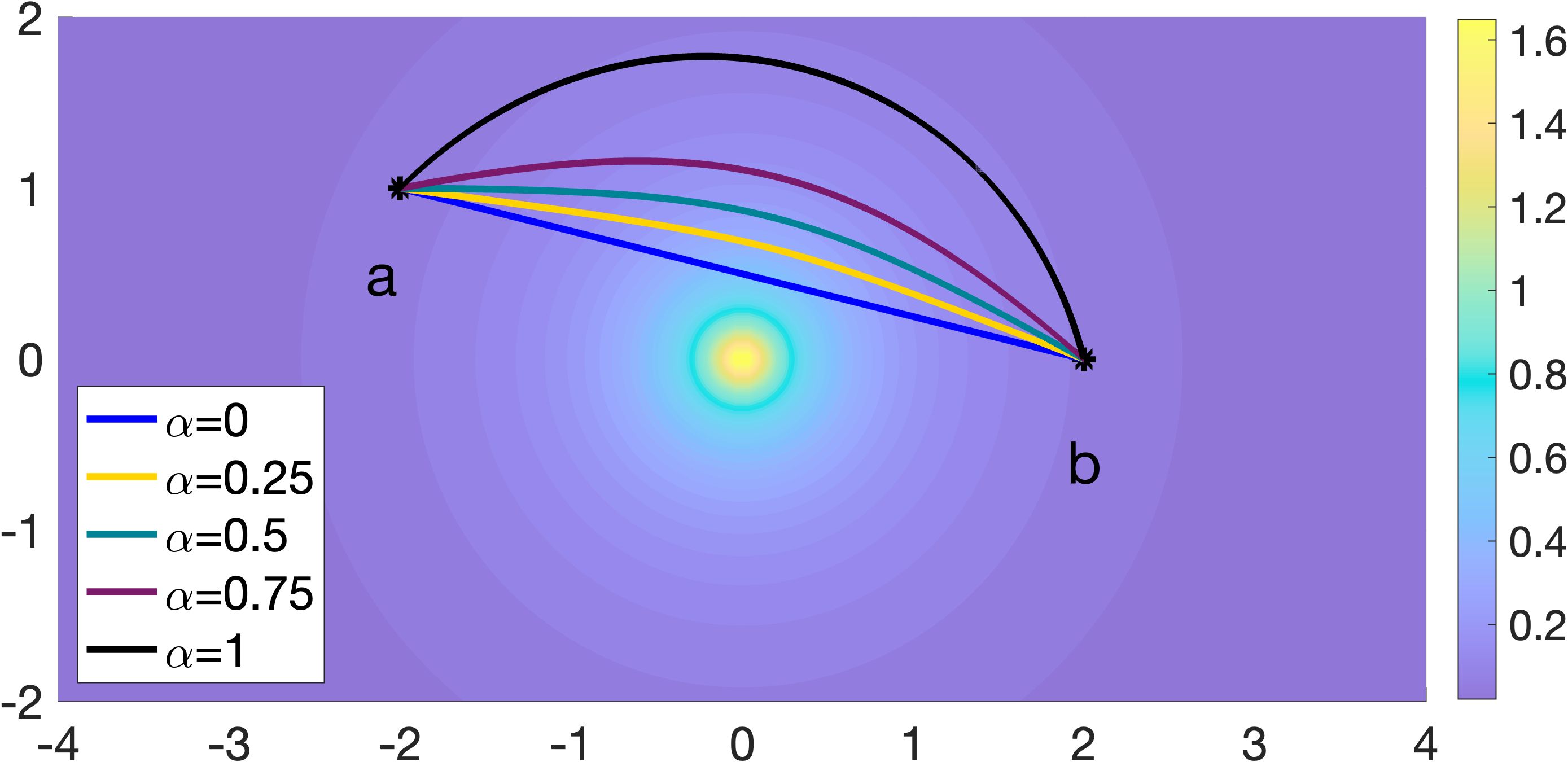}
		\caption{Energy solution}
	\end{subfigure}
	\begin{subfigure}{.49\textwidth}
		\centering
		\includegraphics[width=\linewidth]{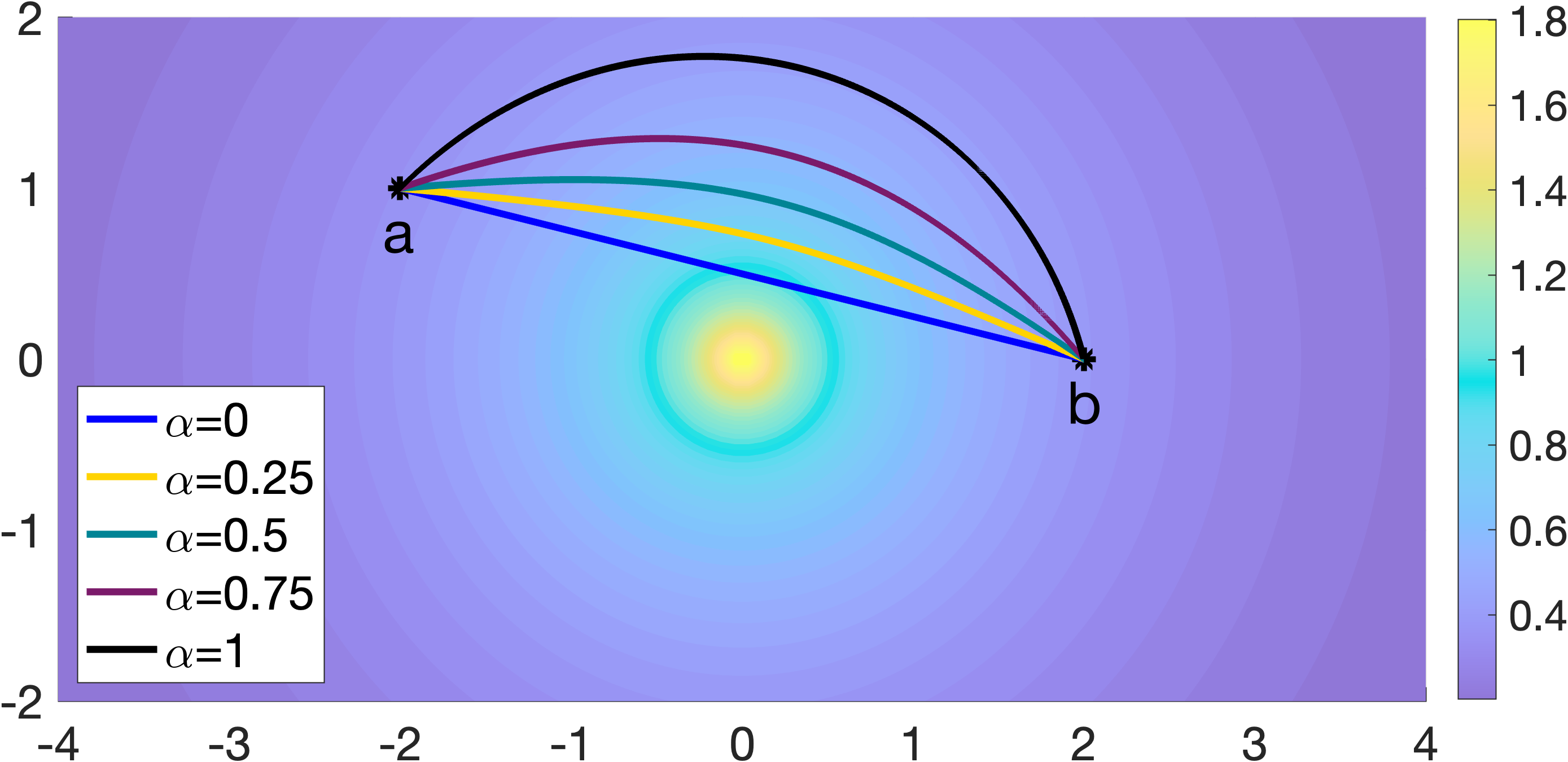}
		\caption{Length solution}
	\end{subfigure}
	\caption{Optimal paths for Example \eqref{E1}}\label{fig:Path_E1}
\end{figure}
\begin{figure}[ht]
	\centering
	\begin{subfigure}{.49\textwidth}
		\centering
		\includegraphics[width=\linewidth]{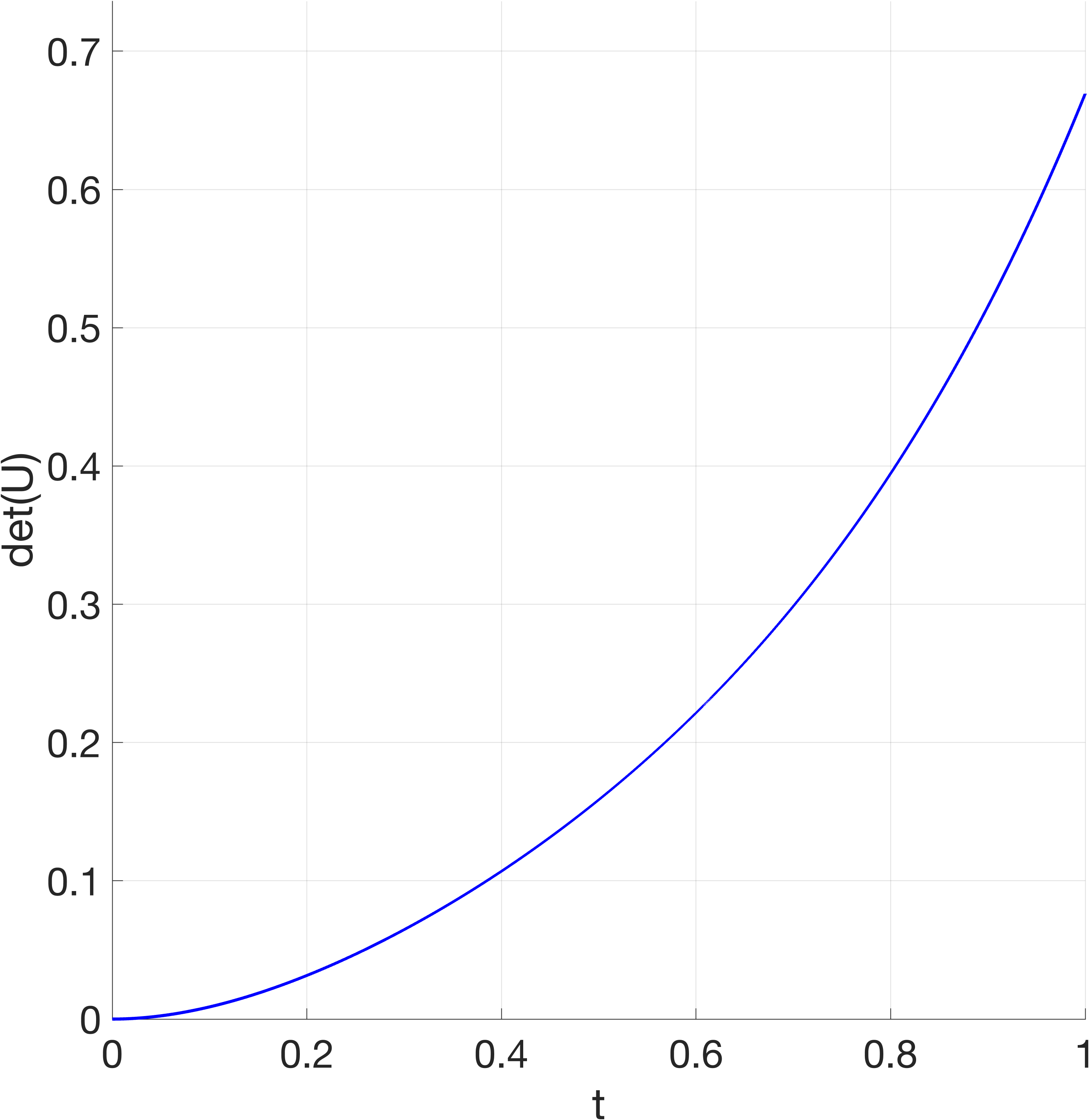}
		\caption{$\det(U)$ values, $\min_k\det(U_k) = 0$ at $t = 0$}
	\end{subfigure}
	\begin{subfigure}{.49\textwidth}
		\centering
		\includegraphics[width=\linewidth]{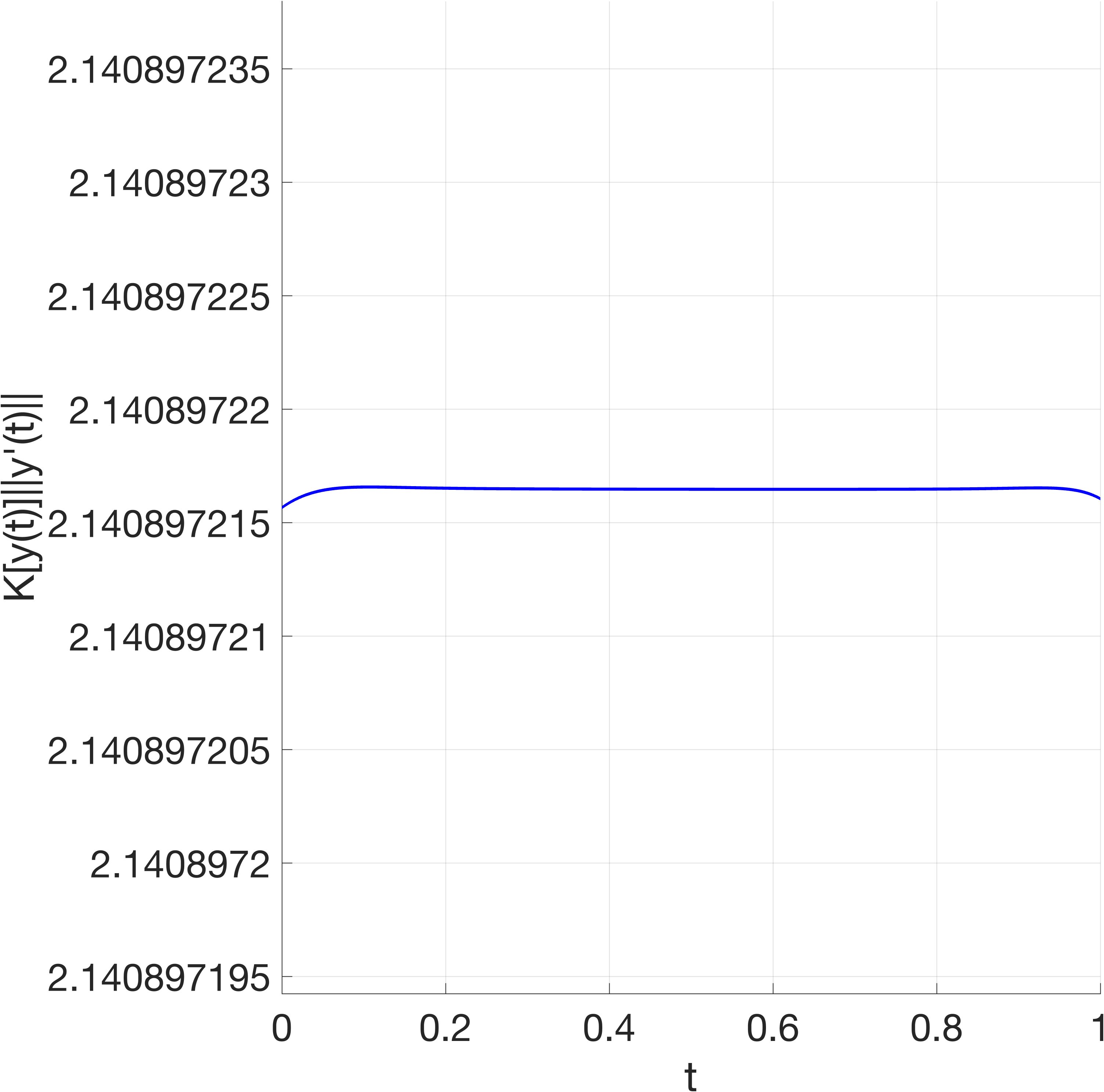}
		\caption{$K(\y) \| \dot \y\|$ values, range is $9.10*10^{-10}$}
	\end{subfigure}
	\caption{Necessary and sufficient conditions for Example \eqref{E1}}\label{fig:Cond_E1}
\end{figure}

\noindent{\bf Example \eqref{E2}}.
Figure \ref{fig:Path_E2} illustrates the solutions for Example \eqref{E2}, while Figure \ref{fig:Cond_E2} confirms that the energy-minimizing path satisfies both necessary and sufficient conditions of optimality. Likewise, the deviation between the squared length and twice the energy is small, as is a comparison of the path lengths:
$$
\left|L(\z)^2 - 2E(\y)\right| = 1.24 \times 10^{-5}\ , \itext{and}
\left|L(\z) - L(\y)\right| = 1.32 \times 10^{-7}\ .
$$

\begin{figure}[ht]
	\centering
	\begin{subfigure}{.49\textwidth}
		\centering
		\includegraphics[width=\linewidth]{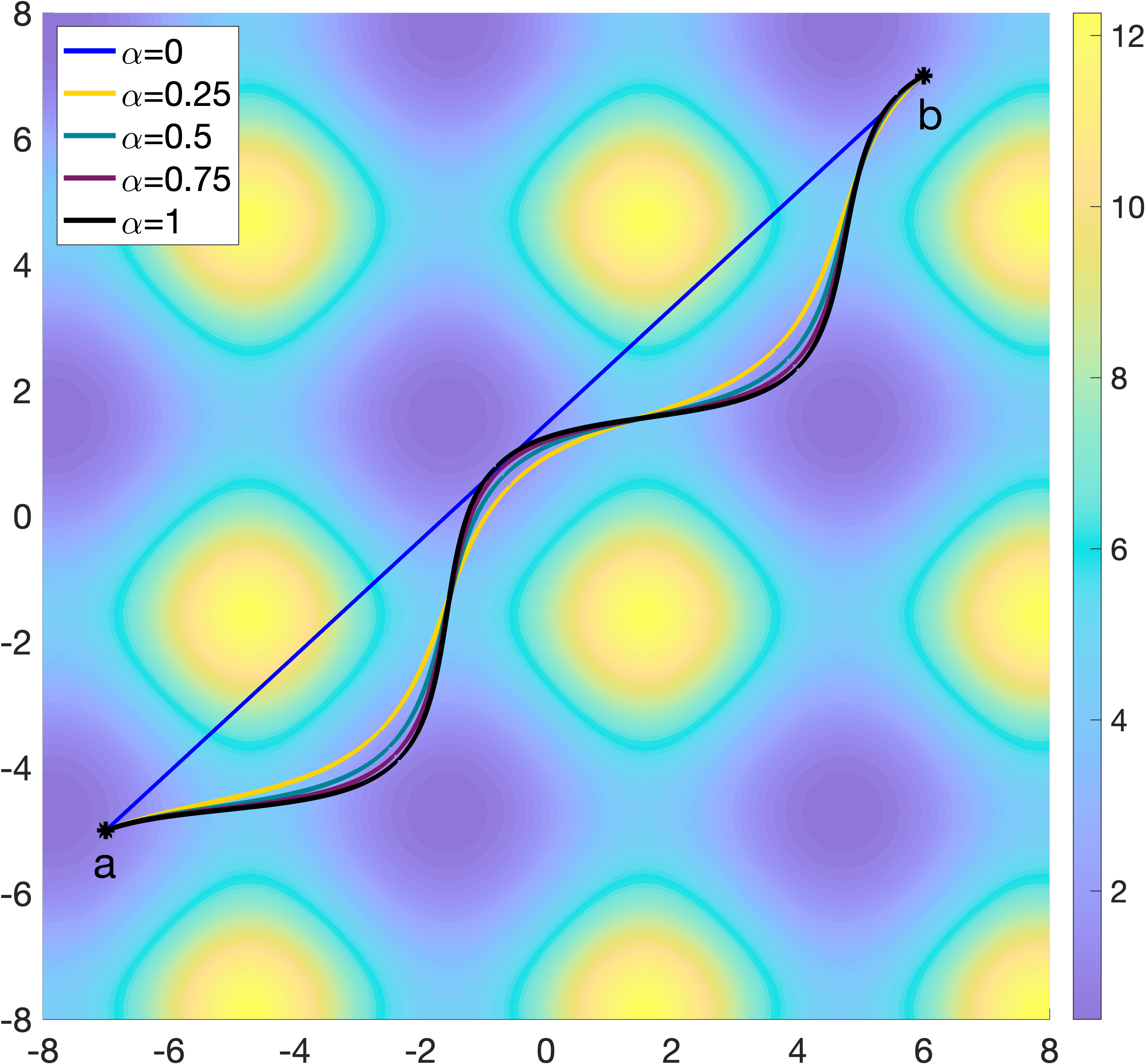}
		\caption{Energy solution}
	\end{subfigure}
	\begin{subfigure}{.49\textwidth}
		\centering
		\includegraphics[width=\linewidth]{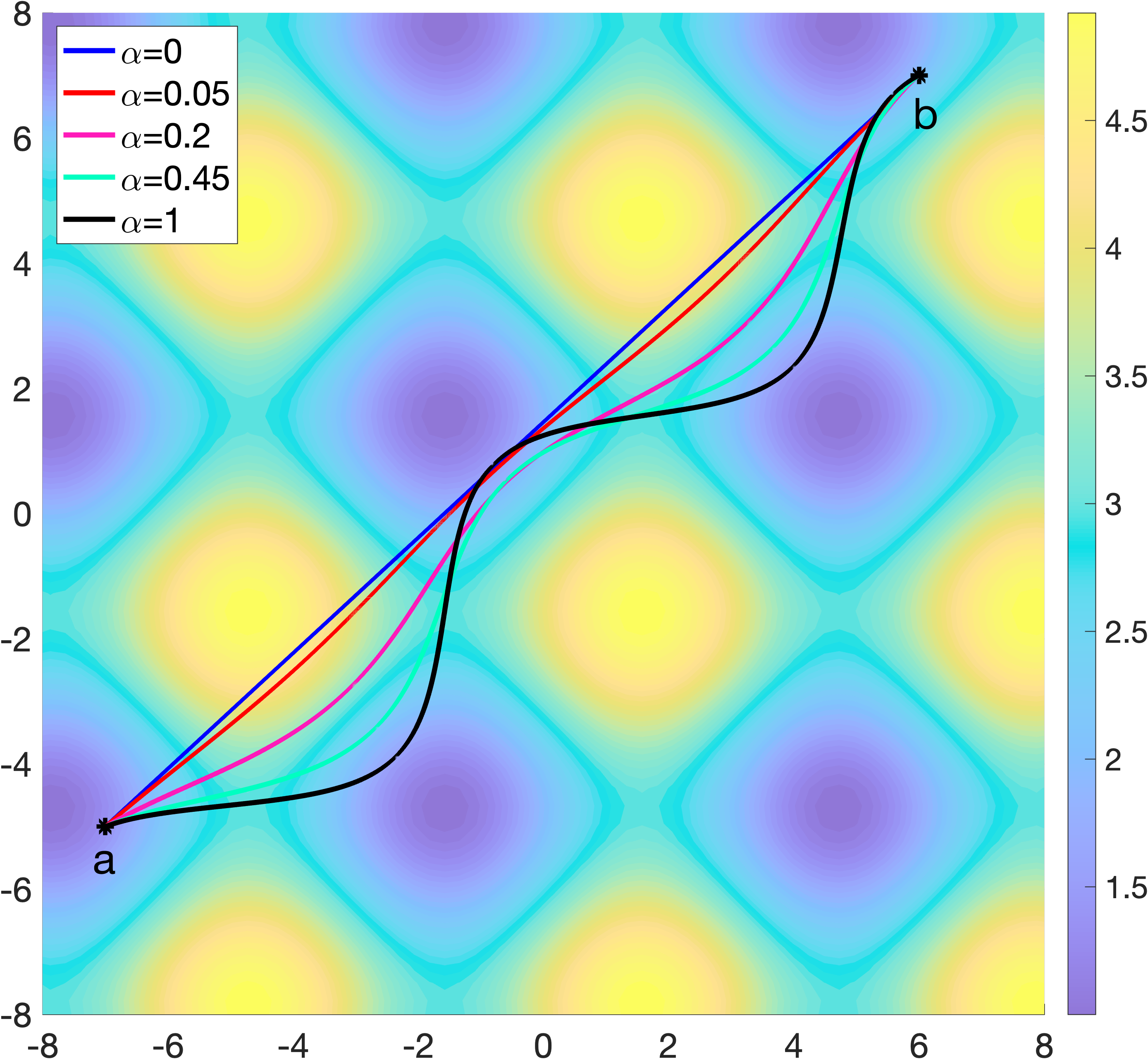}
		\caption{Length solution}
	\end{subfigure}
	\caption{Optimal paths for Example \eqref{E2}}\label{fig:Path_E2}
\end{figure}
\begin{figure}[ht]
	\centering
	\begin{subfigure}{.49\textwidth}
		\centering
		\includegraphics[width=\linewidth]{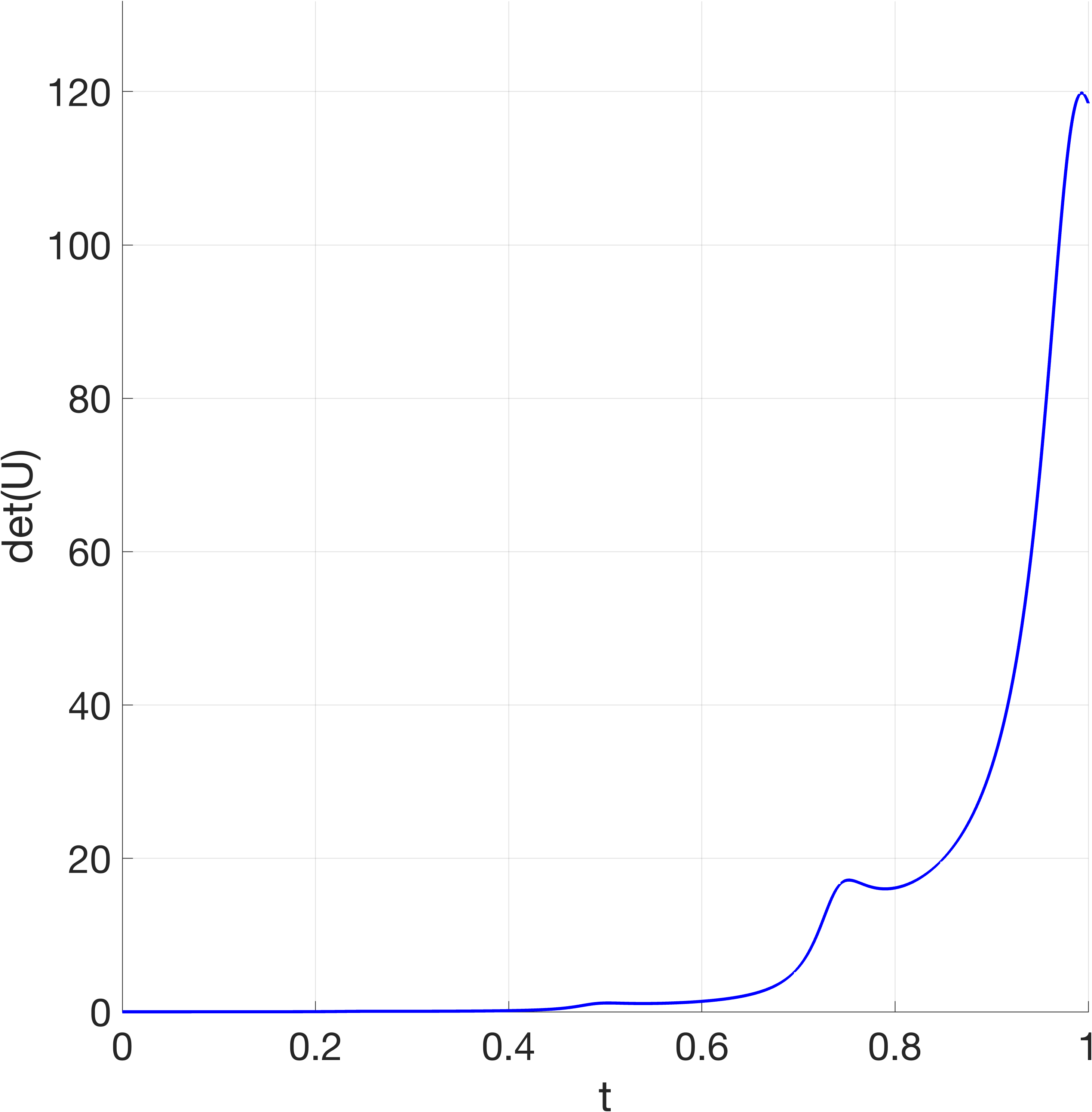}
		\caption{$\det(U)$ values, $\min_k\det(U_k) = 0$ at $t = 0$}
	\end{subfigure}
	\begin{subfigure}{.49\textwidth}
		\centering
		\includegraphics[width=\linewidth]{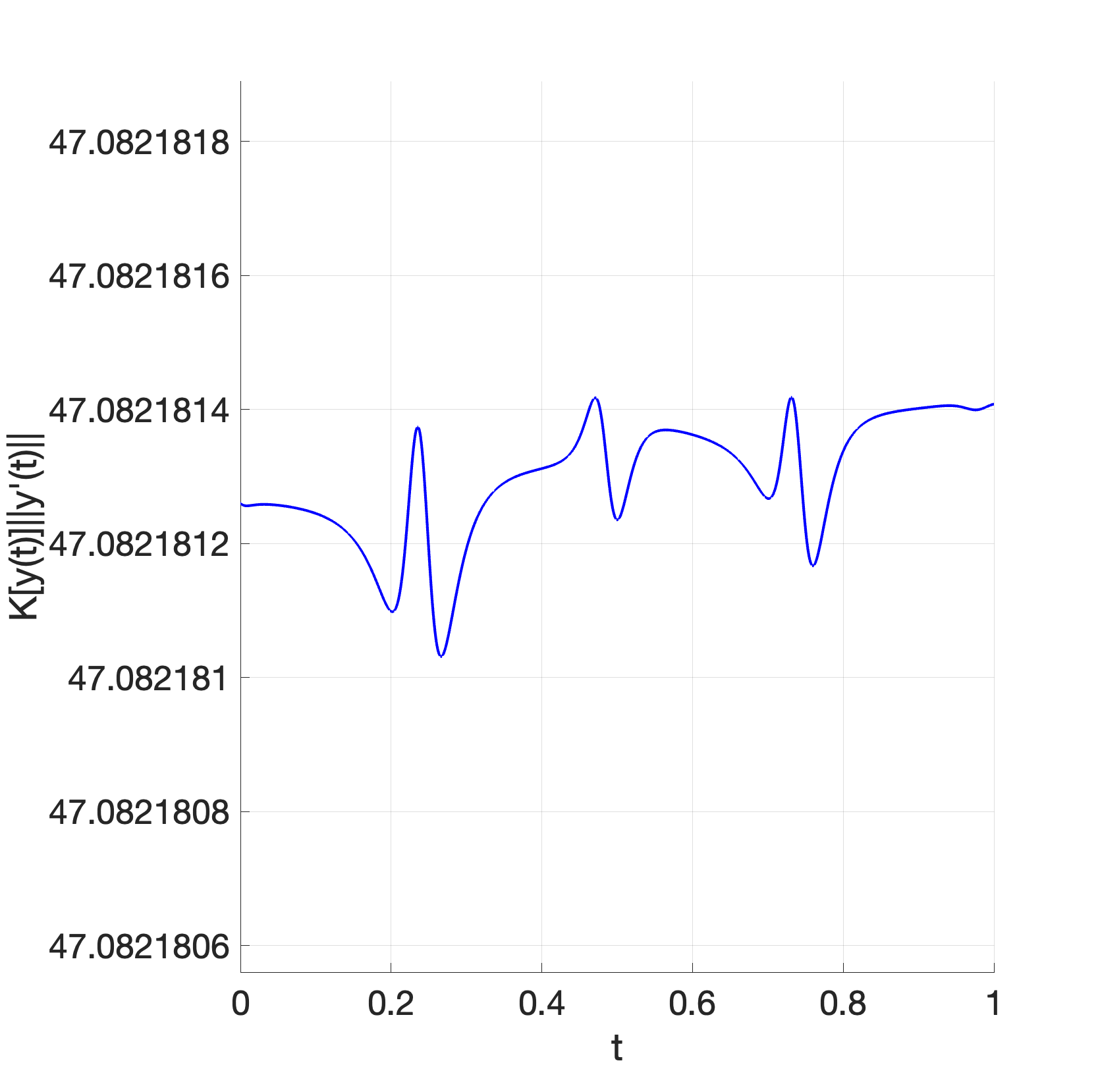}
		\caption{$K(\y) \| \dot \y\|$ values, range is $3.87*10^{-7}$}
	\end{subfigure}
	\caption{Necessary and sufficient conditions for Example \eqref{E2}}\label{fig:Cond_E2}
\end{figure}

\noindent{\bf Example \eqref{E3}}.  Figures \ref{fig:Path_E3} and \ref{fig:Cond_E3} present analogous results in this case. Once again, the energy-minimizing path satisfies necessary and sufficient conditions, and the deviation between the squared length and twice the energy is minimal, as is the difference between the 
lengths of the two solutions:
$$
\left|L(\z)^2 - 2E(\y)\right| = 6.27 \times 10^{-8}\ , \itext{and} 
\left|L(\z) - L(\y)\right| = 1.58 \times 10^{-8}\ .
$$

\begin{figure}[ht]
	\centering
	\begin{subfigure}{.49\textwidth}
		\centering
		\includegraphics[width=\linewidth]{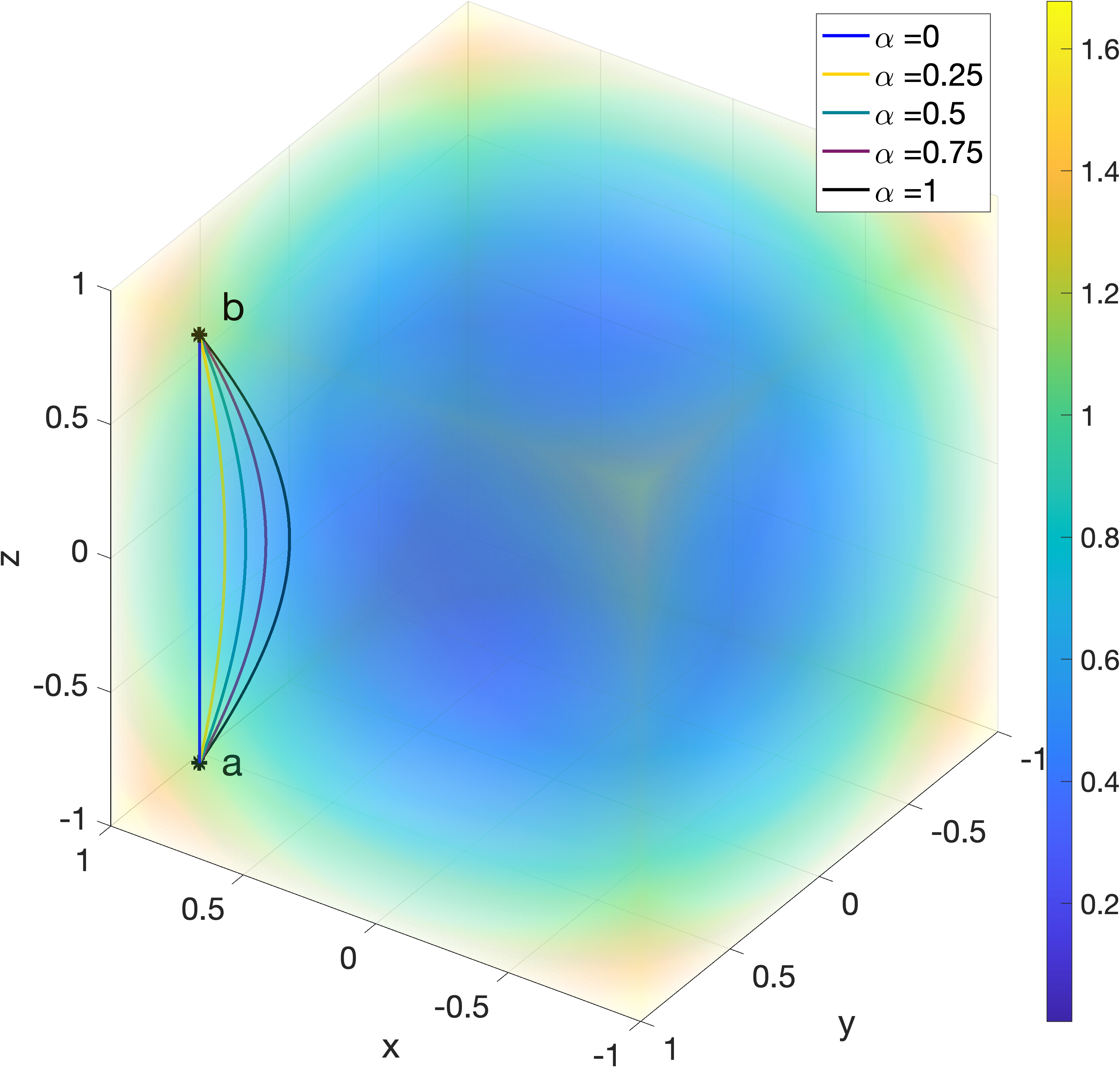}
		\caption{Energy solution}
	\end{subfigure}
	\begin{subfigure}{.49\textwidth}
		\centering
		\includegraphics[width=\linewidth]{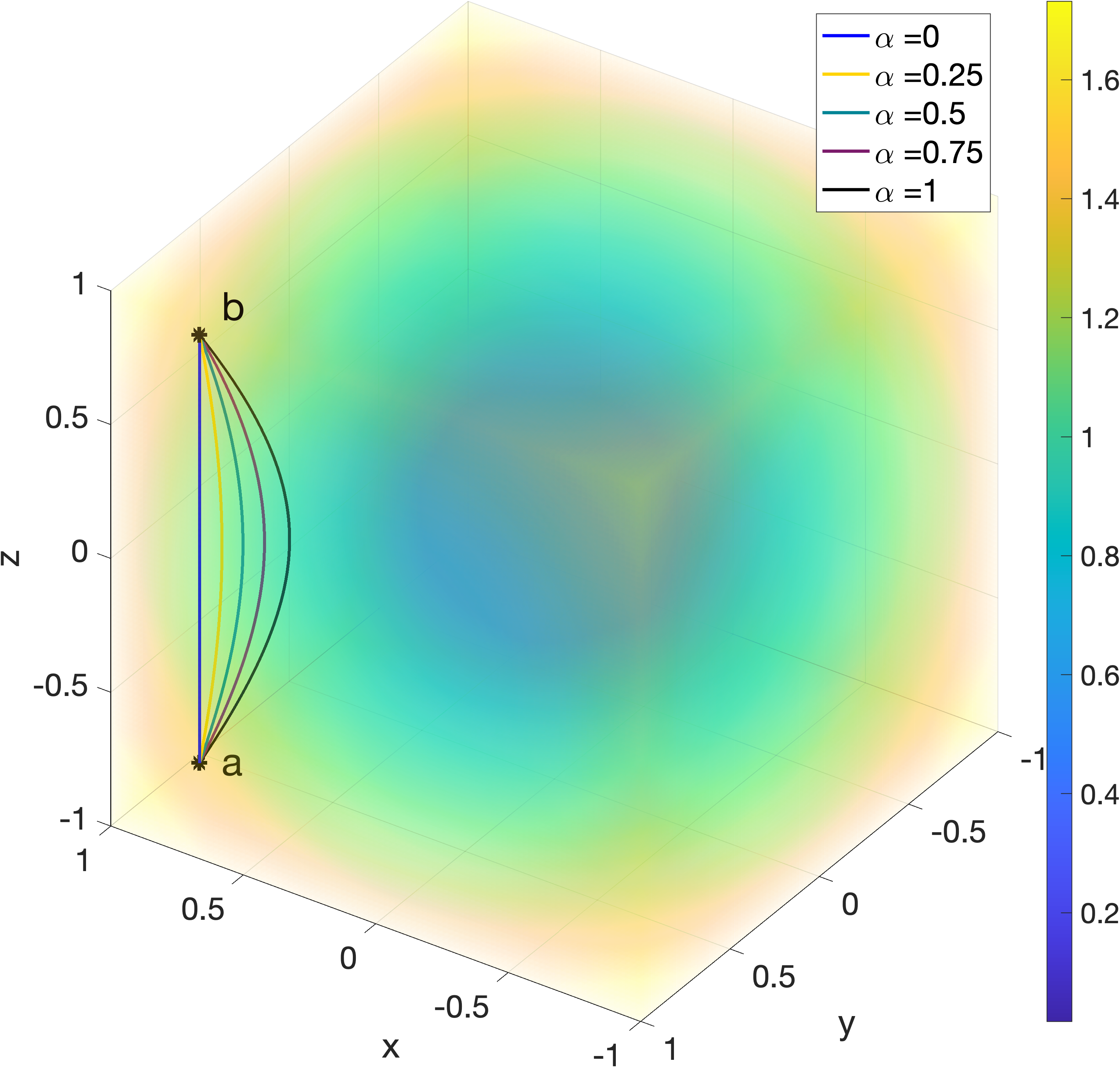}
		\caption{Length solution}
	\end{subfigure}
	\caption{Optimal paths for Example \eqref{E3}}\label{fig:Path_E3}
\end{figure}
\begin{figure}[ht]
	\centering
	\begin{subfigure}{.49\textwidth}
		\centering
		\includegraphics[width=\linewidth]{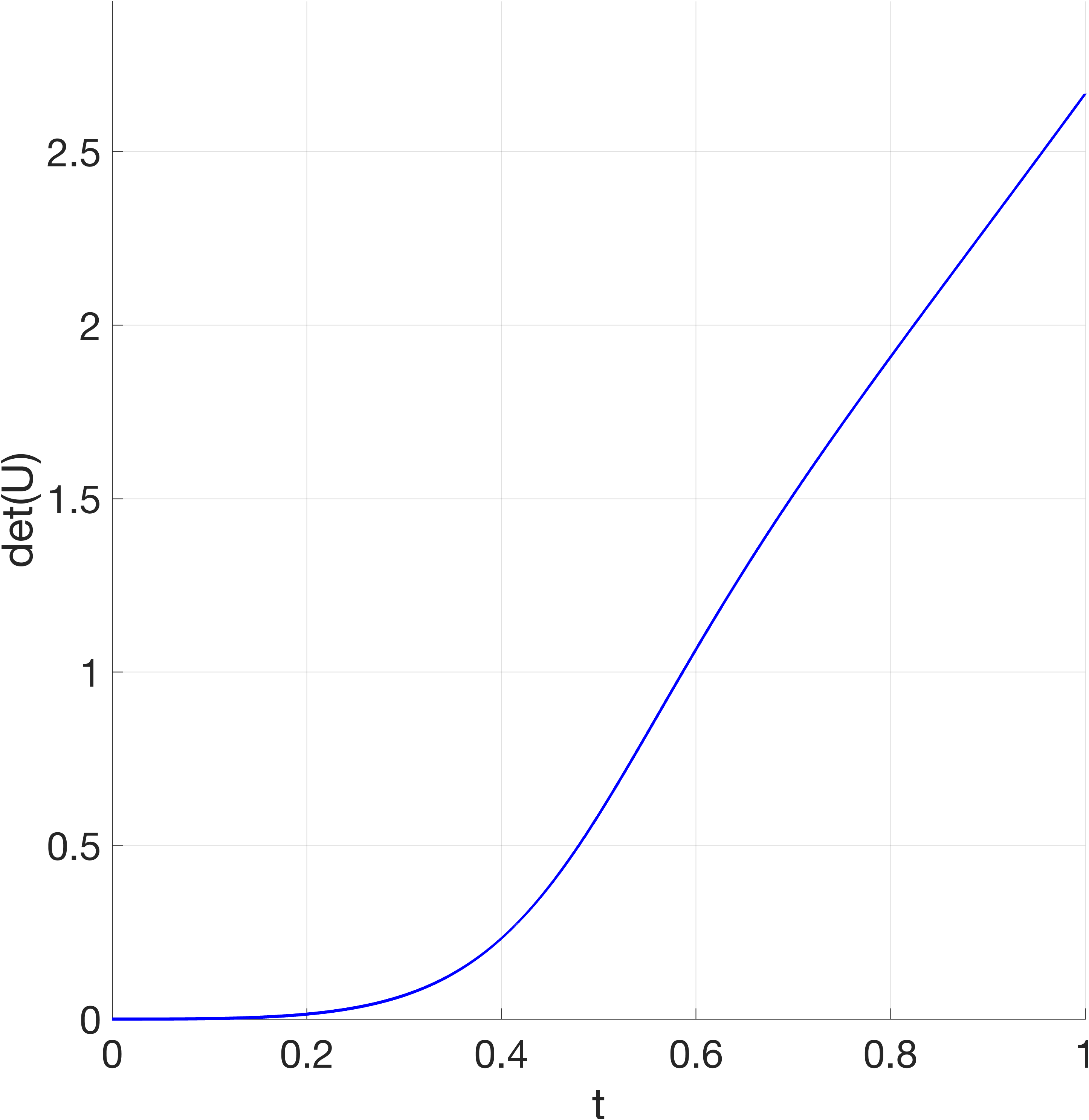}
		\caption{$\det(U)$ values, $\min_k\det(U_k) = 0$ at $t = 0$}
	\end{subfigure}
	\begin{subfigure}{.49\textwidth}
		\centering
		\includegraphics[width=\linewidth]{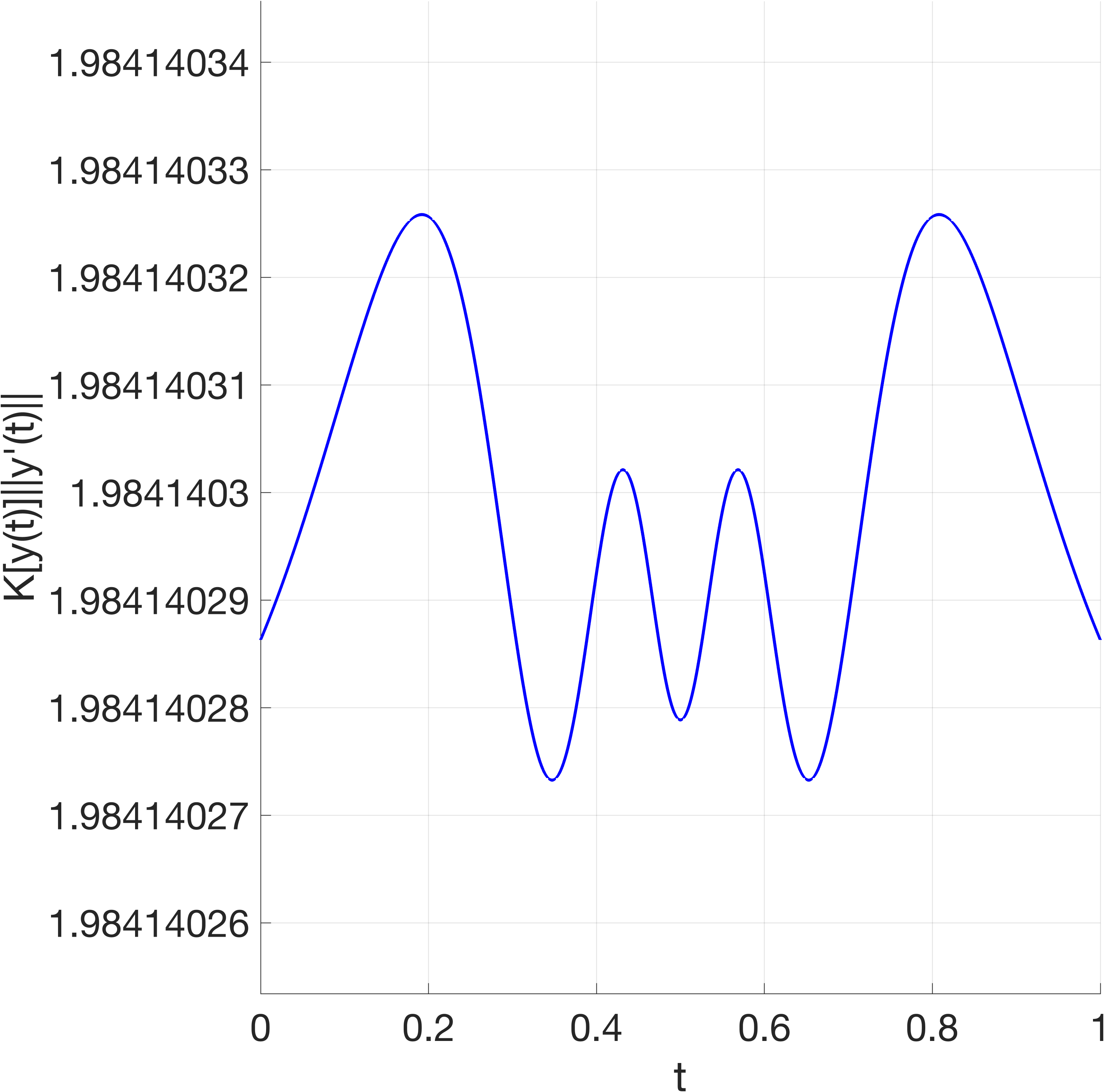}
		\caption{$K(\y) \| \dot \y\|$ values, range is $5.26*10^{-8}$}
	\end{subfigure}
	\caption{Necessary and sufficient conditions for Example \eqref{E3}}\label{fig:Cond_E3}
\end{figure}

\medskip
\subsection{Discrete OT Examples}
Next, we show examples of discrete optimal transport problems which can be solved both by the ``matching algorithm'' and by ``Sinkhorn method".  

It is important to stress that the equivalence between the curves of minimal length and minimal energy we observed in the previous sections does not imply that the optimal solutions obtained when solving the transport problems with one cost matrix, or the other, be the same. In fact, usually, they are not, as we now exemplify, because of the nonlinear relationship between the corresponding cost matrices.

Consider two problems, (\ref{E4} and \ref{E6}), formulated on a uniform mesh grid consisting of 9 points in the two-dimensional case and 8 points in the three-dimensional case with uniform masses. The masses are assumed to be uniform, with $\mu_i = \frac{1}{k}$, $\nu_j = \frac1k$, $\forall i, j$, where $k$ represents the number of points. To compute each entry of the cost matrix, 5 equispaced homotopy steps were used. In rare cases where the boundary value problem solver did not attain the specified tolerance, the number of homotopy steps was increased to a maximum of 51. For the sake of completeness, we point out that each entry of the cost matrix is computed by solving the relevant Euler-Lagrange equation, \eqref{EL-length} or \eqref{EL-energy} and that for the cost given by \eqref{energy-cost}, verification of optimality of all trajectories needed to setup the cost matrix was always successful.

\begin{gather}
    \mathbf{\Omega_X} = \bigg{[}-\frac{31}{10},-\frac{14}{5}\bigg{]}\times \bigg{[}-\frac{9}{10},0\bigg{]}\ ,\ \mathbf{\Omega_Y} = \bigg{[}-\frac94,\frac{13}{4}\bigg{]}\times \bigg{[}\frac14,\frac54\bigg{]}\ ,\ K(\x) = \frac{1}{\frac12+||\x||}\ ,\label{E4}\\
	\begin{split}
		&\mathbf{\Omega_X} = \bigg{[}-\frac{91}{100},-\frac{63}{100}\bigg{]}\times \bigg{[}\frac{63}{100},\frac{91}{100}\bigg{]} \times \bigg{[}-\frac{91}{100},-\frac{63}{100}\bigg{]}\ ,\\
		&\mathbf{\Omega_Y} = \bigg{[}\frac{59}{100},\frac{89}{100}\bigg{]}\times \bigg{[}\frac{59}{100},\frac{89}{100}\bigg{]} \times \bigg{[}\frac{59}{100},\frac{89}{100}\bigg{]}\ ,\ K(\x) = ||\x|| + \frac{1}{10}\ .
	\end{split}\label{E6}
\end{gather}

\subsubsection{Optimal Assignment Algorithm}\label{Opt_Assign}
In Figure \ref{fig:Assign_E4}, we see that the assignment plans obtained for the energy-based and length-based cost matrices in Example \eqref{E4} are the same. However, the assignment plan for Example \eqref{E6} differ, as shown in Figure \ref{fig:Assign_E6}.

\begin{figure}[ht]
	\centering
	\begin{subfigure}{.49\textwidth}
		\centering
		\includegraphics[width=\linewidth]{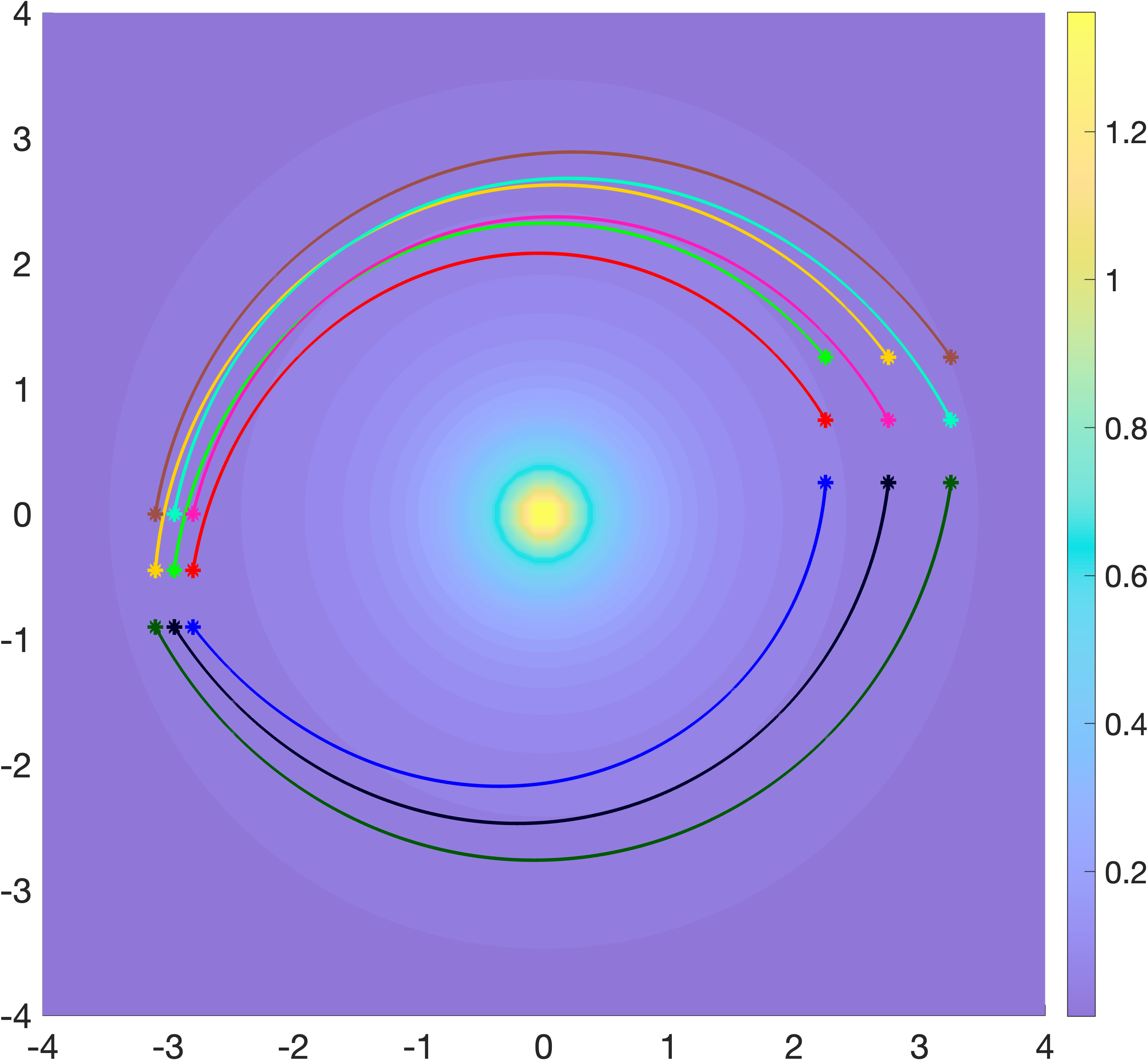}
		\caption{Energy solution, total cost is $2.9975$}
	\end{subfigure}
	\begin{subfigure}{.49\textwidth}
		\centering
		\includegraphics[width=\linewidth]{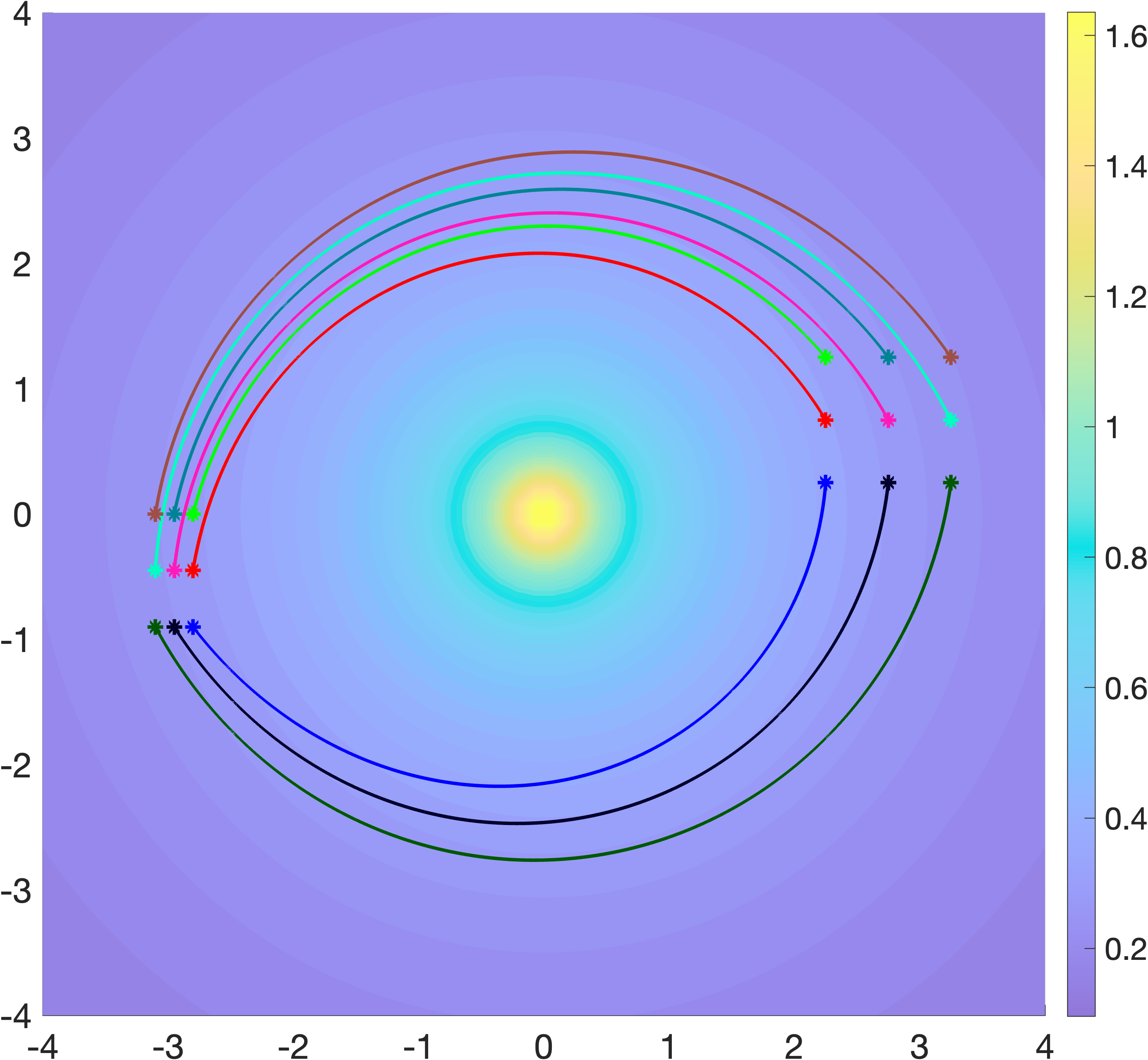}
		\caption{Length solution, total cost is $2.4478$}
	\end{subfigure}
	\caption{Optimal assignment for Example \eqref{E4}}\label{fig:Assign_E4}
\end{figure}

\begin{figure}[ht]
	\centering
	\begin{subfigure}{.49\textwidth}
		\centering
		\includegraphics[width=\linewidth]{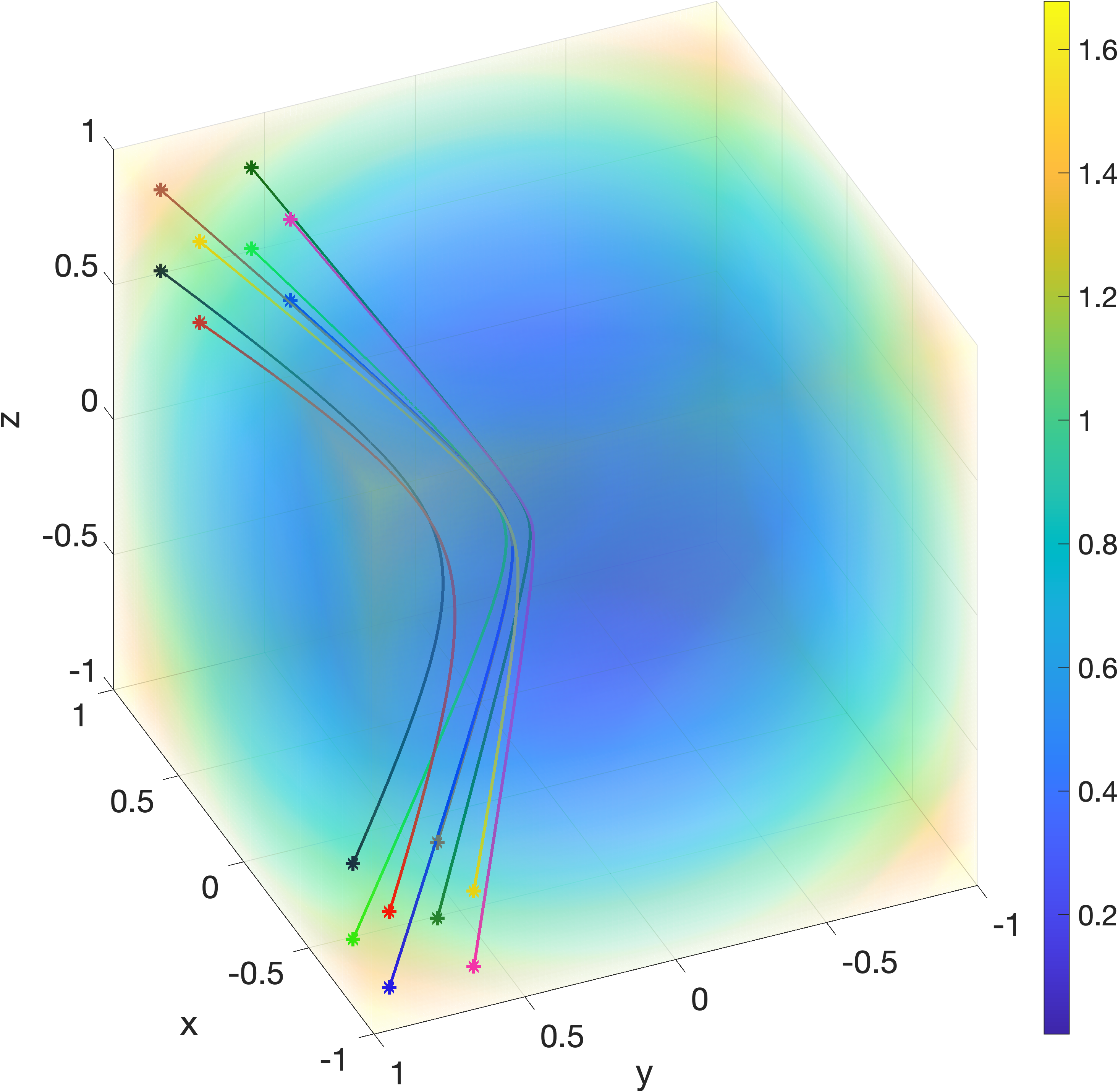}
		\caption{Energy solution, total cost is $2.0574$}
	\end{subfigure}
	\begin{subfigure}{.49\textwidth}
		\centering
		\includegraphics[width=\linewidth]{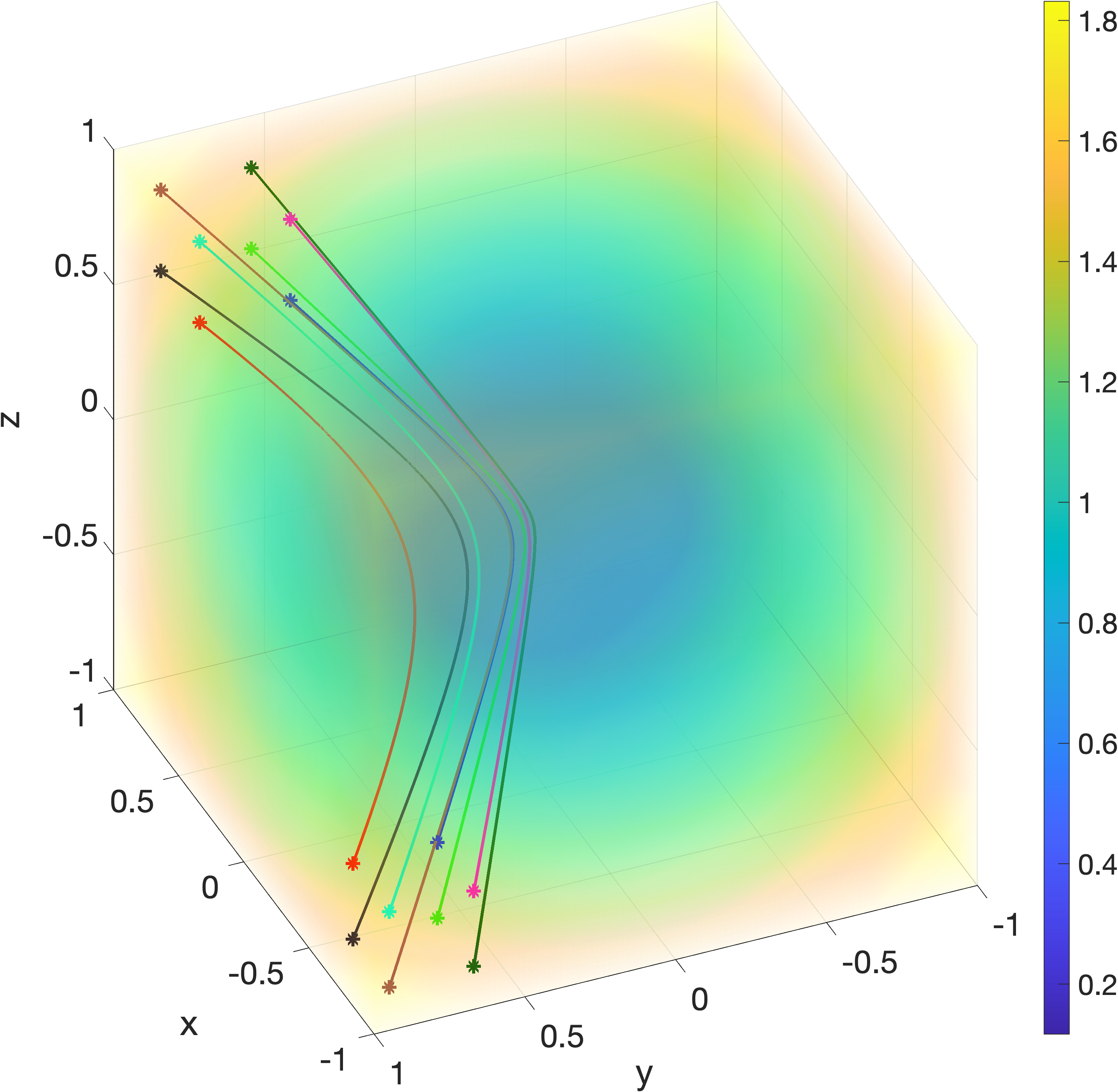}
		\caption{Length solution, total cost is $2.0262$}
	\end{subfigure}
	\caption{Optimal assignment for Example \eqref{E6}}\label{fig:Assign_E6}
\end{figure}

\subsubsection{Linear Programming Formulation: Sinkhorn Method}\label{Sinkhorn}
When using this formulation, beside the aforementioned difference in the transport solutions when using length or energy costs, we also witness a possible mass splitting, as expected.
Figure \ref{fig:Sinkhorn_E4} and Figure \ref{fig:Sinkhorn_E6} illustrate the optimal transport plans for Examples \eqref{E4} and \eqref{E6}. In these visualizations, the brightness of each line corresponds to the amount of mass transported along the respective path. As shown in Figures \ref{fig:Sinkhorn_E4a} and \ref{fig:Sinkhorn_E6a}, the transport plans derived from the energy-based cost matrix are sparser and more closely resemble an assignment mapping. In contrast, when using the length-based cost matrix, the Sinkhorn algorithm produces denser transport matrix. Consequently, the transported mass is more diffusely distributed among the target points, as evident in Figures \ref{fig:Sinkhorn_E4b} and \ref{fig:Sinkhorn_E6b}. These observations further underscore the nonlinear relationship between the cost matrices, whose entries are computed via either energy or length minimization.

Finally, we consider an example where there is an unequal number of number of points and nonuniform densities in the two domains $\mathbf{X}$ and $\mathbf{Y}$.  The setup resembles that of a semi-discrete Optimal Transport problem, and we are going to trace the boundary between different regions of the target space.  We have:
\begin{gather}
	\mathbf{X} = \left\{\icol{-\frac{5}{2} \\ 3},\ \icol{-2 \\ 3},\ \icol{-\frac{3}{2} \\ 3} \right\}\ ,\quad \mathbf{\Omega_Y} = \left[\frac{1}{2}, \frac{5}{2} \right] \times \left[\frac{3}{4}, \frac{11}{4} \right]\ ,\quad K(\mathbf{x}) = \|\mathbf{x}\| + \frac{1}{10}\ , \label{E7}
\end{gather}
where the target points are arranged on a uniform $10 \times 10$ grid over the domain $\mathbf{\Omega_Y}$. The source distribution $\mu$ over $\mathbf{X}$ is nonuniform, given by $\mu = \left(\frac{1}{4},\ \frac{1}{2},\ \frac{1}{4} \right)$, while the target distribution $\nu$ is uniform, with $\nu_j = \frac{1}{100}$ for $j = 1, 2, \dots, 100$.

Figure \ref{fig:Sinkhorn_E7} illustrates the Sinkhorn solution of Example \eqref{E7}, where the color of each target point is proportional to the amount of mass transported from each of the three source points in $\mathbf{X}$. As observed, the energy transport appears sparser, and a rough partitioning of $\mathbf{\Omega_Y}$ into regions influenced by each source point becomes visible. In contrast, the Sinkhorn solution based on the length-minimizing cost matrix results in a significantly more diffuse transport pattern.

\begin{figure}[ht]
	\centering
	\begin{subfigure}{.49\textwidth}
		\centering
		\includegraphics[width=\linewidth]{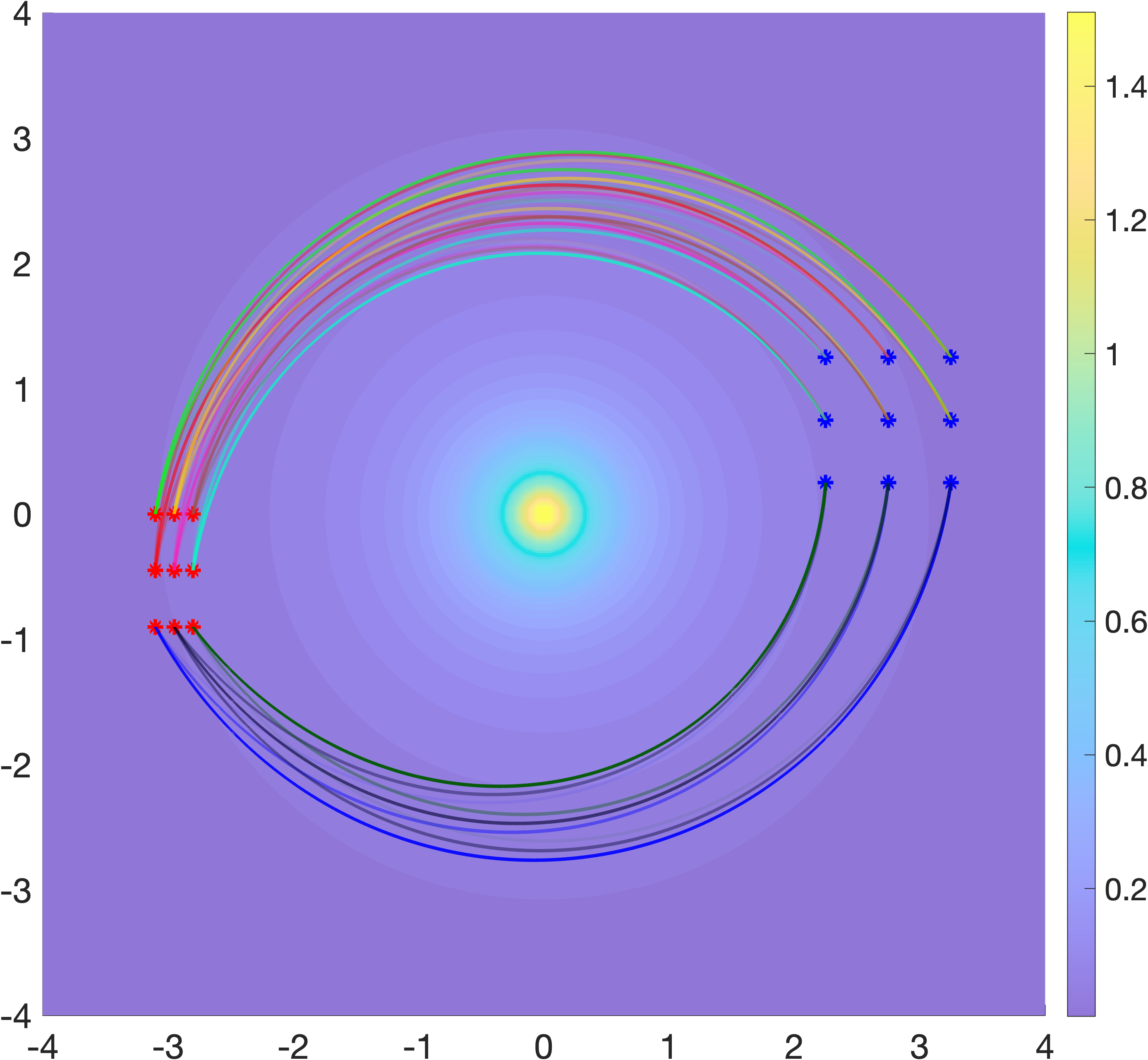}
		\caption{Energy solution, total cost is $2.9996$}
        \label{fig:Sinkhorn_E4a}
	\end{subfigure}
	\begin{subfigure}{.49\textwidth}
		\centering
		\includegraphics[width=\linewidth]{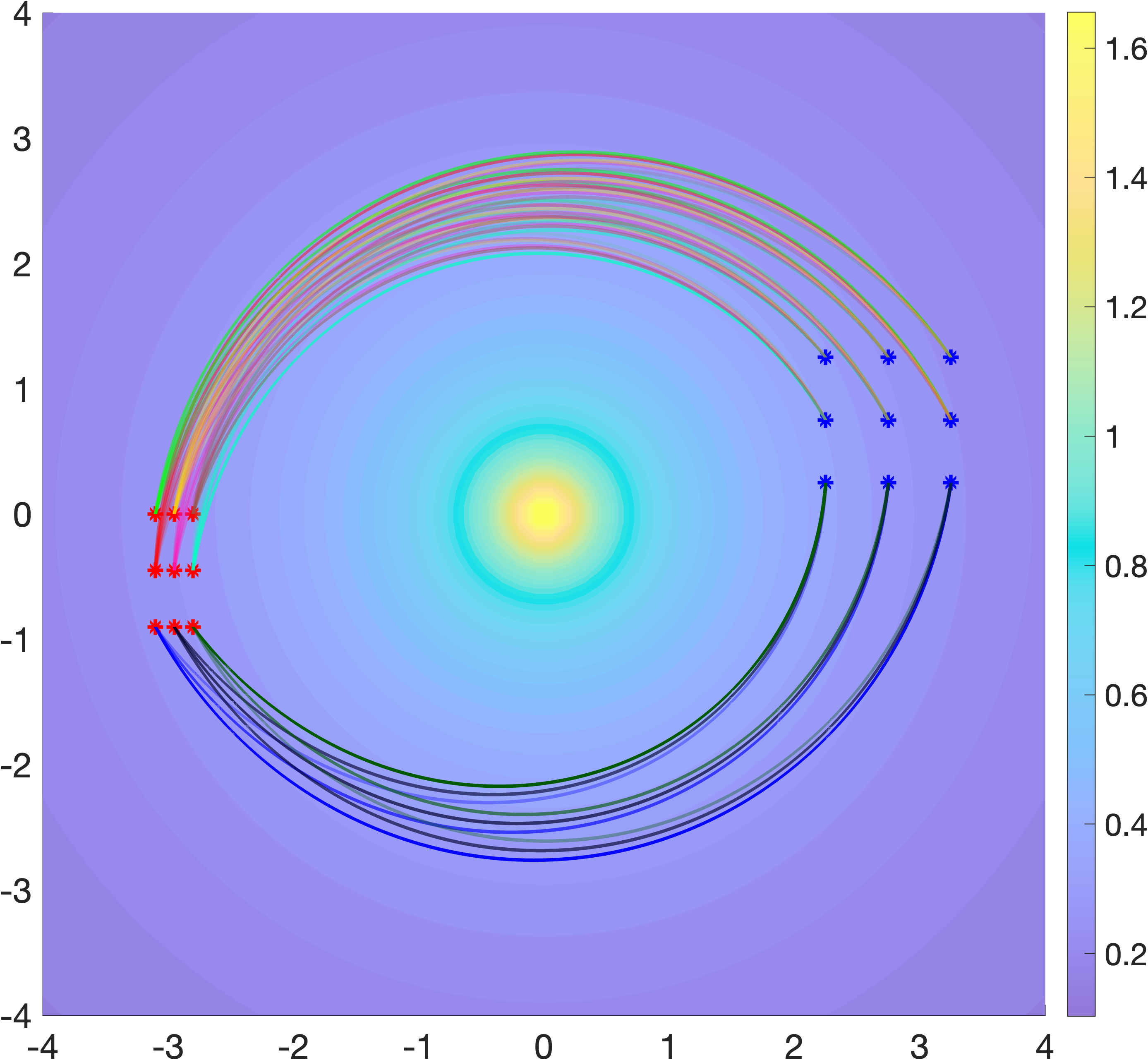}
		\caption{Length solution, total cost is $2.4490$}
        \label{fig:Sinkhorn_E4b}
	\end{subfigure}
	\caption{Sinkhorn solution for Example \eqref{E4}, $\epsilon = \frac{1}{200}$}\label{fig:Sinkhorn_E4}
\end{figure}

\begin{figure}[ht]
	\centering
	\begin{subfigure}{.49\textwidth}
		\centering
		\includegraphics[width=\linewidth]{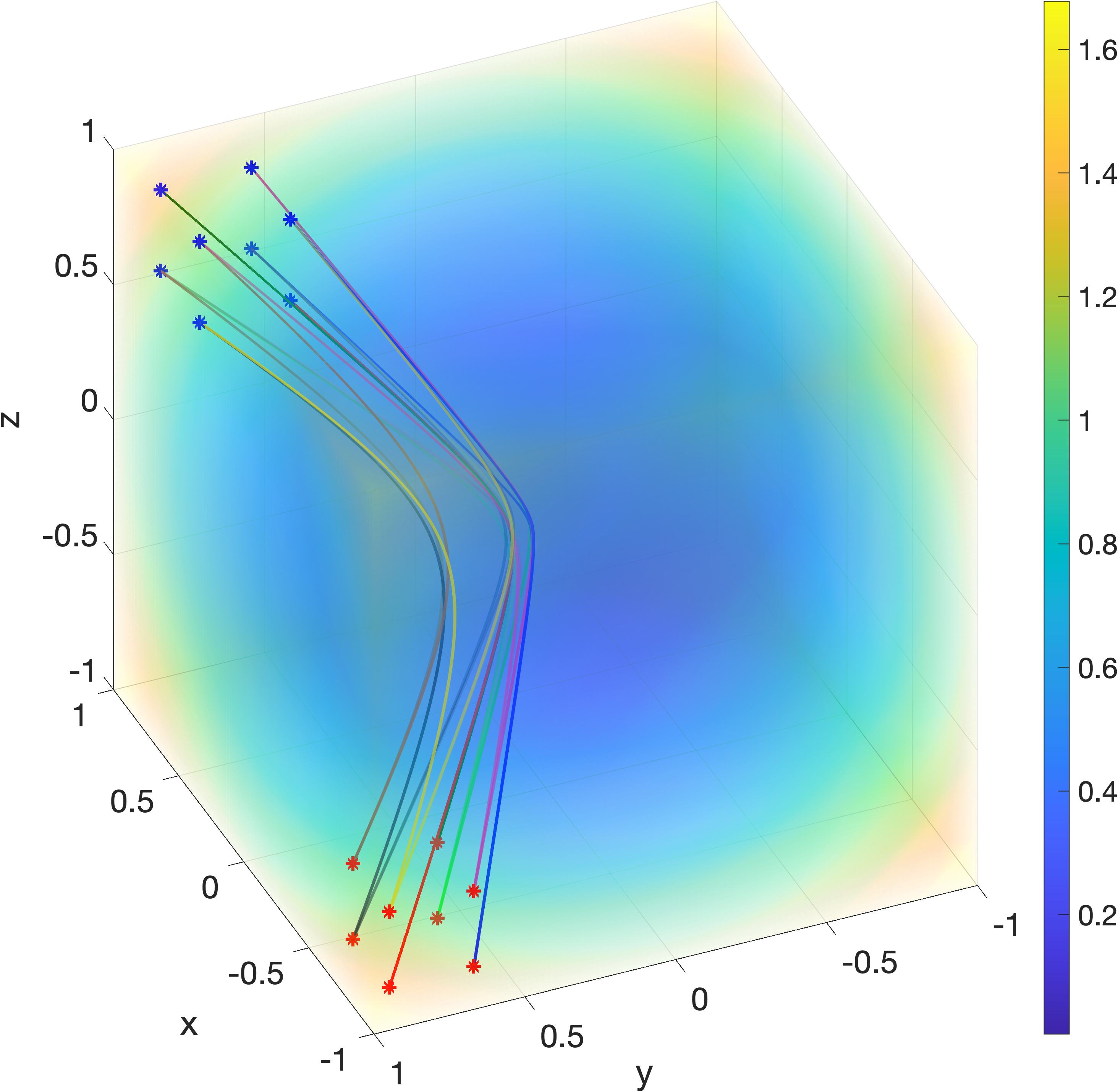}
		\caption{Energy solution, total cost is $2.0577$}
        \label{fig:Sinkhorn_E6a}
	\end{subfigure}
	\begin{subfigure}{.49\textwidth}
		\centering
		\includegraphics[width=\linewidth]{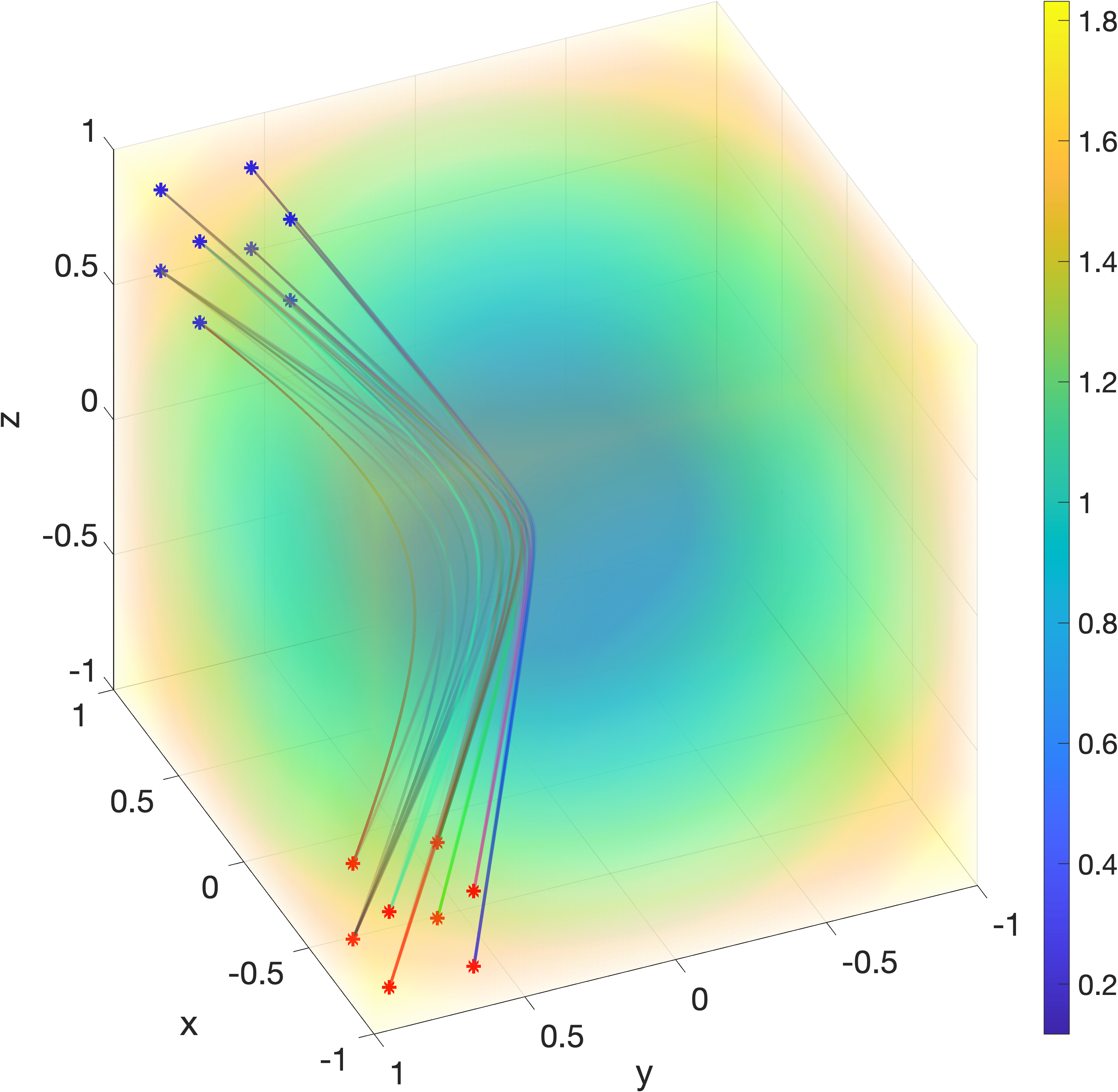}
		\caption{Length solution, total cost is $2.0278$}
        \label{fig:Sinkhorn_E6b}
	\end{subfigure}
	\caption{Sinkhorn solution for Example \eqref{E6}, $\epsilon = \frac{1}{250}$}\label{fig:Sinkhorn_E6}
\end{figure}

\begin{figure}[ht]
	\centering
	\begin{subfigure}{.49\textwidth}
		\centering
		\includegraphics[width=\linewidth]{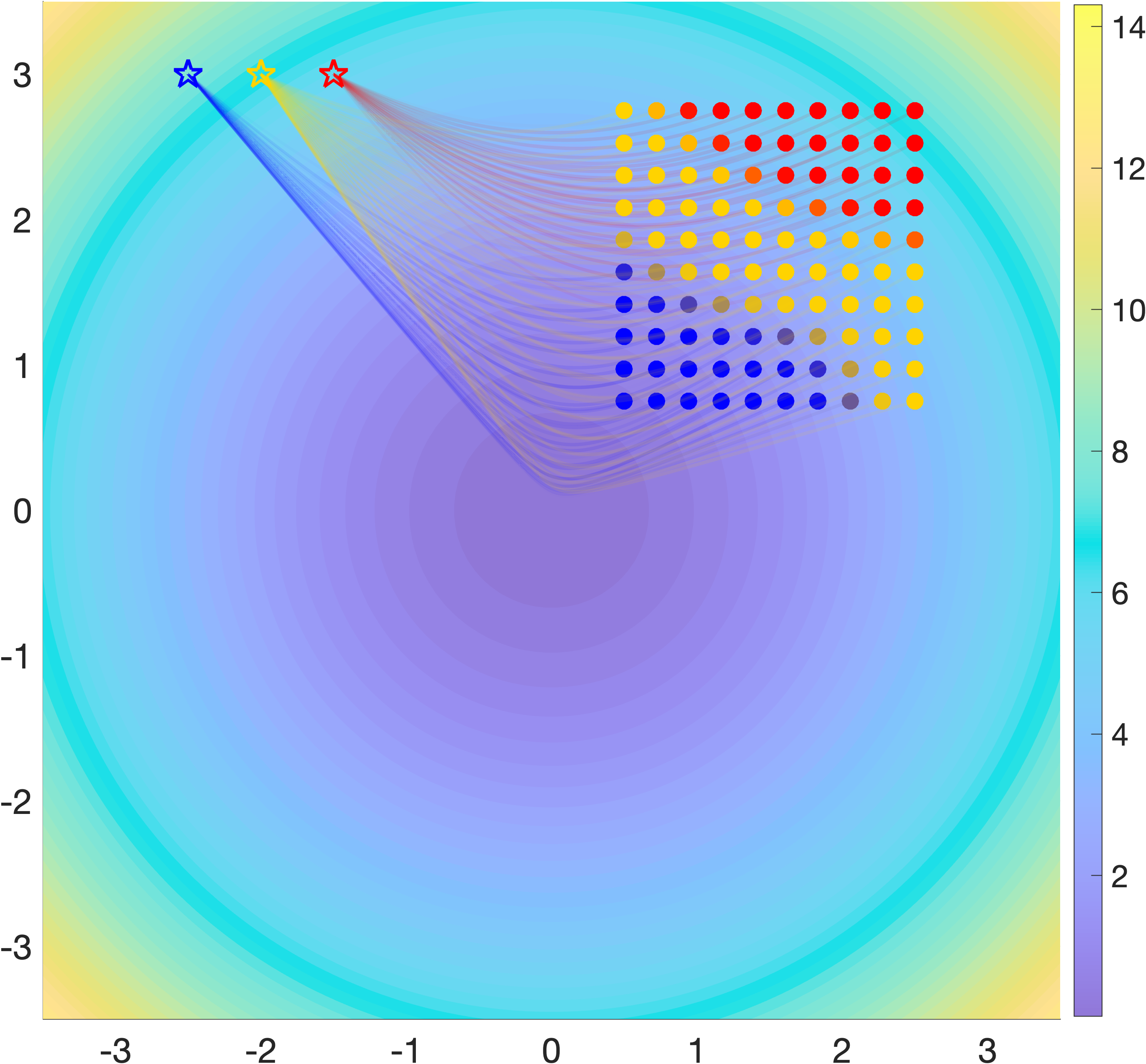}
		\caption{Energy solution, total cost is $44.935$}
        \label{fig:Sinkhorn_E7a}
	\end{subfigure}
	\begin{subfigure}{.49\textwidth}
		\centering
		\includegraphics[width=\linewidth]{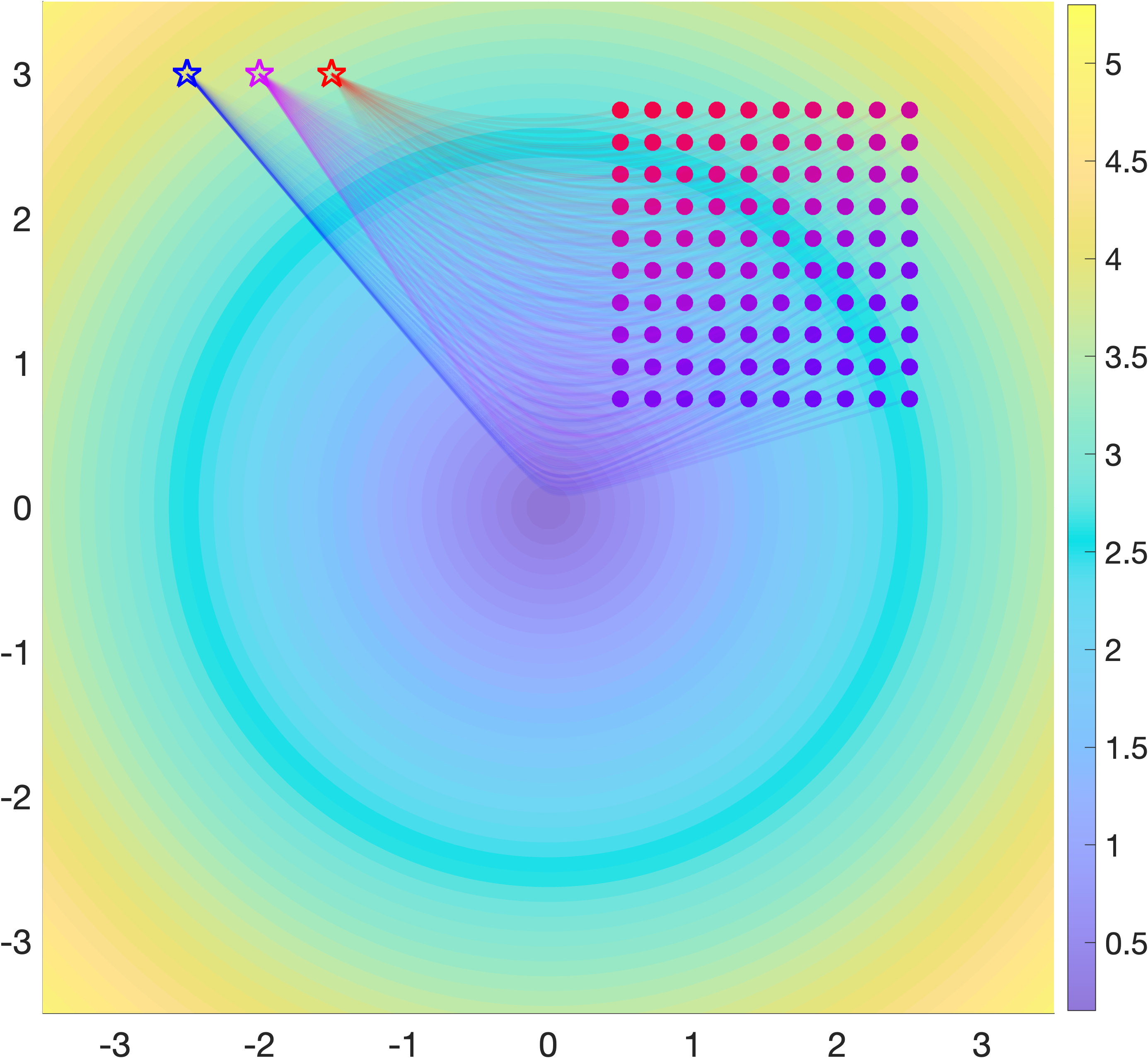}
		\caption{Length solution, total cost is $9.4193$}
        \label{fig:Sinkhorn_E7b}
	\end{subfigure}
	\caption{Sinkhorn solution for Example \eqref{E7}, $\epsilon = \frac{1}{5}$}\label{fig:Sinkhorn_E7}
\end{figure}

\section{Conclusions}\label{Concl}

In this work, our main effort has been directed toward forming the cost matrix underpinning the solution of a discrete optimal transport problem in a nonuniform environment.  From the mathematical modeling point of view, the non-uniformity of the environment reflects in having that the cost of moving one unit of mass from location $\ba$ to location $\bb$, call it $c(\ba, \bb)$,  is a weighted version of the standard length or energy costs; see \eqref{length-cost} and \eqref{energy-cost}.  As a consequence, there is no close expression for $c(\ba,\bb)$, and $c(\ba,\bb)$ must be found as the minimizer among all possible paths joining $\ba$ and $\bb$.   For this latter task, we considered, and solved, the associated Euler-Lagrange equations.   We further provided verifiable sufficient conditions of optimality of the path that we find when solving the Euler-Lagrange equations.  We proposed and implemented our algorithms, and gave evidence of their effectiveness and reliability on several numerical examples in 2 and 3 space dimensions.

\medskip

\subsection*{Declarations}

\noindent -- Ethical approval: Not applicable.\\
\noindent -- Availability of supporting data: Not applicable.\\
\noindent -- Funding: Not applicable.\\
\noindent -- Competing interests: The authors declare no competing interests.\\
\noindent -- Authors' contributions: L.D. and D.O. contributed equally to the manuscript..\\
\noindent -- Acknowledgements: Not applicable.

\subsection*{Acknowledgments}
D.O. gratefully acknowledge that this research was supported in part by the Pacific Institute for the Mathematical Sciences.

\bigskip

\bibliography{reflist}{}
\bibliographystyle{plain}

\end{document}